\documentclass[12pt]{amsart}
\usepackage{fullpage}
\usepackage{amsmath}
\usepackage{amsthm}
\usepackage{amsbsy}
\usepackage{amssymb}
\usepackage{hyperref}
\usepackage{amsfonts}
\usepackage{color}
\usepackage{constants}

\usepackage{enumitem}
\numberwithin{equation}{section}
\newtheorem{theorem}{Theorem}[section]
\newtheorem{lemma}[theorem]{Lemma}
\newtheorem{proposition}[theorem]{Proposition}

\newtheorem*{conjecture*}{Conjecture}
\newtheorem*{hypothesis*}{The generalized Ramanujan conjecture}

\theoremstyle{remark}
\newtheorem*{remark}{Remark}

\theoremstyle{remark}
\newtheorem*{remarks}{Remarks}

\usepackage{amsfonts}
\newtheorem*{theorem*}{Theorem}
\newtheorem{conjecture}[theorem]{Conjecture}
\newcommand{\re}{\textup{Re}}
\newcommand{\im}{\textup{Im}}
\newcommand{\legen}[2]{\genfrac{(}{)}{}{}{#1}{#2}}

\title[The Explicit Sato-Tate Conjecture and Lehmer-Type Questions]{The Explicit Sato-Tate Conjecture and Densities Pertaining to Lehmer-Type Questions}
\author{Jeremy Rouse and Jesse Thorner}
\date{\today}
\begin{document}
\maketitle
\begin{abstract}
Let $f(z)=\sum_{n=1}^\infty a_f(n)q^n\in S^{\text{new}}_ k (\Gamma_0(N))$ be a newform with squarefree level $N$ that does not have complex multiplication.  For a prime $p$, define $\theta_p\in[0,\pi]$ to be the angle for which $a_f(p)=2p^{( k -1)/2}\cos \theta_p $.  Let $I\subset[0,\pi]$ be a closed subinterval, and let $d\mu_{ST}=\frac{2}{\pi}\sin^2\theta d\theta$ be the Sato-Tate measure of $I$.  Assuming that the symmetric power $L$-functions of $f$ satisfy certain analytic properties (all of which follow from Langlands functoriality and the Generalized Riemann Hypothesis), we prove that if $x$ is sufficiently large, then
\[
\left|\#\{p\leq x:\theta_p\in I\} -\mu_{ST}(I)\int_2^x\frac{dt}{\log t}\right|\ll\frac{x^{3/4}\log(N k x)}{\log x}
\]
with an implied constant of $3.33$.  By letting $I$ be a short interval centered at $\frac{\pi}{2}$ and counting the primes using a smooth cutoff, we compute a lower bound for the density of positive integers $n$ for which $a_f(n)\neq0$.  In particular, if $\tau$ is the Ramanujan tau function, then under the aforementioned hypotheses, we prove that
\[
\lim_{x\to\infty}\frac{\#\{n\leq x:\tau(n)\neq0\}}{x}>1-1.54\times10^{-13}.
\]
We also discuss the connection between the density of positive integers $n$ for which $a_f(n)\neq0$ and the number of representations of $n$ by certain positive-definite, integer-valued quadratic forms.
\end{abstract}

\section{Introduction and Statement of Results}

Let $\pi(x)=\#\{p\leq x:\text{$p$ is prime}\}$.  By the modified proof of the prime number theorem due to de la Vall\'ee-Poussin in 1899, there exists a constant $c>0$ such that
\[
\pi(x)=\text{Li}(x)+O\left(xe^{-c\sqrt{\log x}}\right),\hspace{.25in}\text{Li}(x)=\int_2^x \frac{dt}{\log t}.
\]
This result follows from detailed analysis of the nontrivial zeroes of the Riemann zeta function $\zeta(s)$.  In particular, Hadamard and de la Vall\'ee-Poussin proved that if $\rho=\beta+i\gamma$ is a nontrivial zero of $\zeta(s)$ (that is, $\rho$ is in the critical strip $0<\re(s)\leq 1$), then $\beta<1$.  Riemann conjectured that even more is true; he predicted that $\re(\rho)=\frac{1}{2}$ for all nontrivial zeroes $\rho$.  Much evidence favors this hypothesis; for example, Conrey, Iwaniec, and Soundararajan \cite{CIS} proved that over 56\% of all nontrivial zeros of all Dirichlet $L$-functions lie on the line $\re(s)=\frac{1}{2}$.  Assumption of the Riemann Hypothesis improves de la Vall\'ee-Poussin's estimate to
\[
\pi(x)=\text{Li}(x)+O(\sqrt{x}\log x).
\]
Though a proof of the Riemann Hypothesis has eluded mathematicians for the last 150 years, the supposition of its truth has far-reaching consequences.  In this paper, we use an approach similar to that of de la Vall\'ee Poussin to count primes in a different setting.  We will assume the Generalized Riemann Hypothesis, which states that the nontrivial zeroes of other normalized $L$-functions have real part equal to $\frac{1}{2}$.

Let $q=e^{2\pi iz}$ with $\im(z)>0$, and let
\[
f(z)=\sum_{n=1}^{\infty} a_f(n)q^n\in S_k^{\text{new}}(\Gamma_0(N))
\]
be a cusp form of even weight $k\geq2$ and level $N$ such that $a_f(1)=1$ and $f$ is an eigenform of all Hecke operators and of the Atkin-Lehner involutions $\mid_k W(N)$ and $\mid_k W(Q_p)$ for all $p\mid N$.  For brevity, we call such a cusp form a {\it newform}.  (See \cite[Chapter 5, Section 2]{OnoBook} for further details and discussion.)  Furthermore, we require that $f$ not have complex multiplication (CM).  Deligne's proof of the Weil conjectures implies (among many other things) that if $p$ is prime, then there exists an angle $\theta_p\in[0,\pi]$ such that
\[
a_f(p)=2p^{(k-1)/2}\cos \theta_p .
\]
It is natural to consider the distribution of the sequence $\{\theta_p\}$ in the interval $[0,\pi]$.  The Sato-Tate Conjecture, now a theorem due to Barnet-Lamb, Geraghty, Harris, and Taylor \cite{Sato-Tate}, gives us this distribution.
\begin{theorem*}[The Sato-Tate Conjecture]
\label{sato-tate-conjecture}
Let $f(z)\in S_k^{\textup{new}}(\Gamma_0(N))$ be a non-CM newform.  If $F:[0,\pi]\to\mathbb{C}$ is a continuous function, then
\[
\lim_{x\to\infty}\frac{1}{\pi(x)}\sum_{p\leq x}F(\theta_p)=\int_0^\pi F(\theta)~d\mu_{ST},
\]
where $d\mu_{ST}= \frac{2}{\pi}\sin^2 \theta ~d\theta$ is the Sato-Tate measure.  In particular, if we define
\[
\pi_{f,I}(x):=\#\{p\leq x:\theta_p\in I\},
\]
then
\[
\pi_{f,I}(x)\sim \mu_{ST}(I)\mathrm{Li}(x).
\]
\end{theorem*}

\begin{remark}
See \cite{Sato-Tate} for a variant of the Sato-Tate Conjecture that holds for newforms in $S_k^{\mathrm{new}}(\Gamma_0(N),\chi)$ with $\chi$ a nontrivial Dirichlet character.
\end{remark}

The Sato-Tate conjecture is a consequence of analytic properties of symmetric
power $L$-functions.  If we write the degree 2 $L$-function of $f$ as
\[
  L(s,f) = \sum_{n=1}^{\infty} \frac{a_f(n)}{n^{s + \frac{k-1}{2}}}
  = \prod_{p} (1 - \alpha_{p} p^{-s})^{-1} (1 - \beta_{p} p^{-s})^{-1},
\]
then the symmetric power $L$-functions of $f$ are defined by
\[
  L(s, \mathrm{Sym}^n f ) = \prod_{p | N} L_{p}(s, \mathrm{Sym}^n f ) \prod_{p \nmid N} \prod_{j=0}^{n}
  (1 - \alpha_{p}^{j} \beta_{p}^{n-j} p^{-s})^{-1},
\]
and the Sato-Tate conjecture is equivalent to the statement
that these $L$-functions have no zeroes or poles on the line $s = 1$.  If $p\nmid N$, then we have that $\alpha_p=e^{i\theta_p}$ and $\beta_p=e^{-i\theta_p}$.


When $f\in S_2^{\textup{new}}(\Gamma_0(N))$ is the newform associated to a non-CM elliptic curve $E/\mathbb{Q}$ of conductor $N$, it follows from the work of V. K. Murty \cite{Murty} and Bucur and Kedlaya \cite{BK} that if all of the symmetric power $L$-functions of $f$ are automorphic and satisfy the Generalized Riemann Hypothesis, then
\begin{equation}
\label{murty}
\pi_{f,I}(x)=\mu_{ST}(I)\mathrm{Li}(x)+O\left(x^{3/4}\sqrt{\log(N x)}\right).
\end{equation}
Bucur and Kedlaya \cite{BK} extend this result to arbitrary motives with applications toward elliptic curves over totally real fields.  In particular, let $K$ be a totally real number field with ring of integers $\mathcal{O}_K$, and let $E_1/K$ and $E_2/K$ be two $\bar{\mathbb{Q}}$-nonisogenous  non-CM  elliptic curves with respective conductors $N_1\neq N_2$.  For $i=1,2$ and $\mathfrak{p}\subset\mathcal{O}_K$ a prime ideal, let $a_{E_i}(\mathfrak{p})=\mathrm{Norm}(\mathfrak{p})+1-\#E_i(\mathcal{O}_K/\mathfrak{p})$.  Bucur and Kedlaya apply their results to show that there is a prime ideal $\mathfrak{p}$ with $\text{Norm}(\mathfrak{p})\nmid N_1 N_2$ and $\text{Norm}(\mathfrak{p})=O((\log N_1N_2)^2(\log\log 2N_1N_2)^2)$ at which $a_{E_1}(\mathfrak{p})a_{E_2}(\mathfrak{p})<0$.  This problem is an elliptic curve analogue of the least quadratic nonresidue problem.

In this paper, we prove a completely explicit version of the Sato-Tate Conjecture that applies to every newform of even weight $k\geq2$ on $\Gamma_0(N)$ with $N$ squarefree, in which case the newforms in question are guaranteed to be non-CM.  Furthermore, we improve the error term
in \eqref{murty} by a factor of $\sqrt{\log x}$. We assume the following analytic hypotheses.

\begin{conjecture}
\label{automorphy-GRH}
Let $f\in S_k^{\textup{new}}(\Gamma_0(N))$ be a non-CM newform with squarefree level $N$ and even integer weight $k\geq2$.  For each integer $n\geq0$, the following are true.
\begin{enumerate}[label=(\alph*)]
\item The conductor of $ L(s,\mathrm{Sym}^n f) $ is $q_{\mathrm{Sym}^{n} f} = N^{n}$.
\item The equation for the gamma factor of $L(s,\mathrm{Sym}^n f)$ is
\[
\gamma(s, \mathrm{Sym}^n f )=
\begin{cases}
\displaystyle\prod_{j=1}^{(n+1)/2}\Gamma_{\mathbb{C}}(s+(j-1/2)( k -1))  &\mbox{if $n$ is odd,} \\
\displaystyle\Gamma_{\mathbb{R}}(s+r)\prod_{j=1}^{n/2}\Gamma_{\mathbb{C}}(s+j( k -1)) & \mbox{if $n$ is even,}
\end{cases}
\]
where $\Gamma_{\mathbb{R}}(s) = \pi^{-s/2} \Gamma(s/2)$, $\Gamma_{\mathbb{C}}(s)
= 2 (2 \pi)^{-s} \Gamma(s)$, and $r = \frac{n}{2} \bmod 2$. ($\Gamma(s)$ denotes
the usual gamma function.)
\item For each prime $p | N$, $L_{p}(s, \mathrm{Sym}^n f ) = (1 - (-\lambda_{p} p^{-1/2})^{n} p^{-s})^{-1}$, where $\lambda_{p}\in\{-1,1\}$ is eigenvalue of the Atkin-Lehner operator
$W(p)$ acting on $f$.
\item The $L$-function $L(s,\mathrm{Sym}^n f)$ satisfies the functional equation
\[
  \Lambda(s,\mathrm{Sym}^n f) := q_{ \mathrm{Sym}^n f }^{s/2} \gamma(s,\mathrm{Sym}^n f)
   L(s,\mathrm{Sym}^n f)  = \epsilon_{ \mathrm{Sym}^n f } \Lambda(1-s,\mathrm{Sym}^n f),
\]
where $\epsilon_{\mathrm{Sym}^n f}$ is a certain complex number of modulus 1.
\item The completed $L$-function $(\frac{1}{2}s(1-s))^{\delta_{n,0}}\Lambda(s,\mathrm{Sym}^n f)$ is entire, where $\delta_{n,0}$ is $1$ if $n=0$ and $0$ otherwise.
\item The Generalized Riemann Hypothesis (GRH) holds for each symmetric power $L$-function; that is, for each $n\geq0$, each zero of $\Lambda(s, \mathrm{Sym}^n f )$ has real part equal to $\frac{1}{2}$.
\end{enumerate}
\end{conjecture}

It is reasonable to assume parts (a)-(e) of Conjecture \ref{automorphy-GRH} because Langlands functoriality predicts that $L(s,\mathrm{Sym}^n f)$ is the $L$-function of a
cuspidal automorphic representation $\Pi$ on
$\mathrm{GL}_{n+1}(\mathbb{A}_{\mathbb{Q}})$. More specifically, the newform $f$ corresponds to an
automorphic representation $\pi$ on $\mathrm{GL}_{2}(\mathbb{A}_{\mathbb{Q}})$ and
$\pi$ has a tensor product decomposition $\pi = \otimes_{p \leq
  \infty} \pi_{p}$. It is conjectured that there is an automorphic
representation $\Pi = \otimes_{p \leq \infty} \Pi_{p}$, where $\Pi_{p}
= {\rm Sym}^{n}(\pi_{p})$ is the representation of $\mathrm{GL}_{n+1}(\mathbb{Q}_{p})$
that is the local symmetric $n$-th power lift of $\pi_{p}$ (which
exists and is unique via the local Langlands
correspondence). Moreover, it is conjectured that if $f$ does not have
CM, then $\Pi$ is cuspidal.

Assuming a global lifting map on automorphic representations that is
compatible with the local Langlands correspondence, Cogdell and Michel \cite{CM}
compute the predicted equations for the conductor, gamma factor, and
root number of $L(s,\mathrm{Sym}^n f) $. Such a global lifting map is known
unconditionally for $n = 1, 2, 3$ and $4$ (see \cite{GJ, Kim, KS1, KS2}), and Conjecture~\ref{automorphy-GRH} (with the exception of GRH) is known under those
assumptions.

Our main result is the following.
\begin{theorem}
\label{main-theorem}
Let $f(z)=\sum_{n=1}^{\infty}a_f(n)q^n\in S_k^{\textup{new}}(\Gamma_0(N))$ be a newform which satisfies Conjecture \ref{automorphy-GRH}, and let $I=[\alpha,\beta]\subset[0,\pi]$.  If $x\geq 2$ and $\mathfrak{q}(f)=N(k-1)$, then
\[
|\pi_{f,I}(x) -\mu_{ST}(I)\mathrm{Li}(x)|\leq3.33x^{3/4}-\frac{3x^{3/4}\log\log x}{\log x}+\frac{202x^{3/4}\log\mathfrak{q}(f)}{\log x}.
\]
\end{theorem}
\begin{remarks}
1.  If $\log x\geq \mathfrak{q}(f)^{68}$, then
\[
|\pi_{f,I}(x) -\mu_{ST}(I)\mathrm{Li}(x)|\leq3.33x^{3/4}.
\]

2.  In \cite{AIW}, Arias-de-Reyna, Inam, and Wiese use this result to prove that
if $I \subseteq [0,\pi]$, then the set of primes $\mathcal{P}_{f,I}
= \{ p : \theta_p \in I \}$ is regular; that is, there
exists a function $g(z)$, holomorphic in $\{ z : {\rm Re}(z) \geq 1 \}$ so that
\[
  \sum_{p \in \mathcal{P}_{f,I}} \frac{1}{p^{z}} = \mu_{ST}(I) \log\left(\frac{1}{1-z}\right) + g(z).
\]
\end{remarks}

Using effective versions of the Sato-Tate Conjecture with a reasonably small error term, one can study the distribution of primes $p$ for which $a_f(p)=c$ for some fixed $c\in\mathbb{R}$.  To study such primes, note that if $c$ is fixed and
\[
a_f(p)=2p^{\frac{k-1}{2}}\cos \theta_p =c,
\]
then as $p$ grows, $\theta_p$ tends to $\frac{\pi}{2}$.  Therefore, studying the primes $p$ for which $\theta_p$ is close to $\frac{\pi}{2}$ allows us to produce bounds on the quantity
\begin{equation}
\label{pi-zero}
\pi_{c,f}(x)=\#\{x<p\leq 2x:a_f(p)=c\}.
\end{equation}
Lang and Trotter \cite{Lang-Trotter} predicted that for some constant $K_{c,f}\geq0$, depending on the fixed number $c$ and the newform $f$, we have
\begin{equation}
\label{eqn:LT}
\pi_{c,f}(x)\sim\begin{cases}
\displaystyle K_{c,f}\frac{\sqrt{x}}{\log x}&\mbox{if $k=2$},\\
K_{c,f}&\mbox{if $k\geq 4$}.
\end{cases}
\end{equation}

Serre \cite{Ser1} studied $\pi_{c,f}(x)$ extensively using variants of the effective versions of the Chebotarev Density Theorem proven by to Lagarias and Odlyzko \cite{LO}. Serre proved that for any $\epsilon>0$, we have
\[
\pi_{c,f}(x)\ll\begin{cases}
x(\log x)^{-3/2+\epsilon}&\mbox{unconditionally},\\
x^{7/8}(\log x)^{-1/2}&\mbox{if $c\neq0$ and GRH holds for Artin $L$-functions},\\
x^{3/4}&\mbox{if $c=0$ and GRH holds for Artin $L$-functions.}
\end{cases}
\]
Under GRH for Artin $L$-functions, M. R. Murty, V. K. Murty, and S. Saradha \cite{MMS} improved on the above result, showing that if $c\neq0$, then
\[
\pi_{c,f}(x)\ll\frac{x^{4/5}}{(\log x)^{1/5}}.
\]
Building on the ideas in \cite{MMS}, V. K. Murty \cite{Murty2} improved Serre's unconditional result, showing that for all $c$, we have
\[
\pi_{c,f}(x)\ll \frac{x(\log\log x)^2}{(\log x)^2}.
\]

Stronger unconditional results exist when $c=0$ and $f(z)$ is the newform associated to a non-CM elliptic curve $E/\mathbb{Q}$, in which case $k=2$.  By the modularity theorem, the primes $p\geq5$ such that $p\nmid N$ and $a_f(p)=0$ are the primes for which $E$ has supersingular reduction.  Noam Elkies \cite{Elkies-1} proved that infinitely many such primes exist.  By the work of Elkies, Fouvry, Kaneko, and M. R. Murty \cite{Elkies-2, FM}, when $f$ is the newform of such an elliptic curve $E/\mathbb{Q}$, we have that for any $\epsilon>0$,
\[
\frac{\log\log\log x}{(\log\log\log\log x)^{1+\epsilon}}\ll \pi_{0,f}(x)\ll x^{3/4}.
\]

Assuming Conjecture \ref{automorphy-GRH}, we consider the primes $p$ for which $\theta_p$ is close to $\frac{\pi}{2}$, bounding the number of primes $x<p\leq 2x$ for which $|a_f(p)|$ is small.  This enables us to give improved upper bounds on $\pi_c(x)$ for all $c$.

\begin{theorem}
\label{count-primes-zero}
Let $f(z)=\sum_{n=1}^{\infty}a_f(n)q^n\in S_k^{\textup{new}}(\Gamma_0(N))$ be a newform which satisfies Conjecture \ref{automorphy-GRH}.  If $x\geq 2$ and $\mathfrak{q}(f)=N(k-1)$, then
\[
\#\{x<p\leq 2x: |a_f(p)|\leq 0.86p^{\frac{k-1}{2}-\frac{1}{4}}\sqrt{\log p}\}
\]
is bounded above by
\[
\frac{9.3 x^{3/4}}{\sqrt{\log x}}-\frac{9.2 x^{3/4} \log \log x}{(\log x)^{3/2}}+\frac{50 x^{3/4}}{(\log x)^{3/2}}+43 \sqrt{x}+(\log \mathfrak{q}(f)) \left(\frac{19 x^{3/4}}{(\log x)^{3/2}}+\frac{22 \sqrt{x}}{\log x}\right).
\]
In particular, if
\[
x\geq 2+\begin{cases}
5\left(\frac{c^2}{\log(c^2+1)}\right)^{\frac{2}{2k-3}}&\mbox{if $c\neq0$,}\\
1&\mbox{if $c=0$},
\end{cases}
\]
then $\pi_{c,f}(x)$ also satisfies this bound.
\end{theorem}

\begin{remarks}
1.  Under the same conditions as Theorem \ref{count-primes-zero}, the first author proved in \cite{Rouse} that if $k\geq4$ and $0\leq\alpha\leq\frac{1}{8}$, then
\[
\#\{x<p\leq 2x:|a_f(p)|\leq p^{\frac{k-1}{2}-\alpha}\}\asymp\frac{x^{1-\alpha}}{\log x}.
\]
The proof of Theorem \ref{count-primes-zero} makes the work in \cite{Rouse} completely explicit.

2.  By counting primes with a smooth cutoff, Zywina \cite{Zywina} recently sharpened the aforementioned arguments of Serre and M. R. Murty, V. K. Murty, and Saradha to show that under the Generalized Riemann Hypothesis for Artin $L$-functions, one has that
\[
\pi_{c,f}(x)\ll\begin{cases}
x^{4/5}(\log x)^{-3/5}&\mbox{if $c\neq0$},\\
x^{3/4}(\log x)^{-1/2}&\mbox{if $c=0$}.
\end{cases}
\]
Thus it appears that the current upper bounds for $\pi_{c,f}(x)$ cannot be improved without the input of a fundamentally new idea.
\end{remarks}

For certain newforms, very strong forms of the Lang-Trotter Conjecture are suspected to hold.  For example, consider the  newform $f(z)\in S_{12}^{\text{new}}(\Gamma_0(1))$ given by
\[
\Delta_{12}(z)=q\prod_{n=1}^{\infty}(1-q^n)^{24}=\sum_{n=1}^\infty \tau_{12}(n)q^n,
\]
where $\tau_{12}(n)$ is the Ramanujan tau function.  In \cite{Leh}, D. H. Lehmer pondered whether $\tau_{12}(n)\neq0$ for all $n\geq1$; this question remains open, as do similar conjectures for other newforms on $\Gamma_0(1)$.  For $k\in\{12,16,18,20,22,26\}$, the space $S_k^{\mathrm{new}}(\Gamma_0(1))$ is one-dimensional and is spanned by
\[
\Delta_k(z)=\Delta_{12}(z)E_{k-12}(z)=\sum_{n=1}^\infty \tau_k(n)q^n,
\]
where $E_k(z)$ is the weight $k$ Eisenstein series (see Section 2).
\begin{conjecture}
\label{Lehmer}
For all $k\in\{12,16,18,20,22,26\}$ and all positive integers $n$, $\tau_k(n)\neq0$.  In other words, for the aforementioned values of $k$, we have $K_{0,\Delta_k}=0$ in \eqref{eqn:LT}.
\end{conjecture}

Serre \cite{Ser1} proved that if $f(z)$ is a non-CM newform, then
\begin{equation}
\label{Serre-density}
\lim_{x\to\infty}\frac{\#\{n\leq x:a_f(n)\neq0\}}{x}=\alpha_f \prod_{a_f(p)=0}\left(1-\frac{1}{p+1}\right),
\end{equation}
where $\alpha_f\in(0,1]$ is a constant which is given in the proof of Theorem 16 of \cite{Ser1}.  When $f(z)\in S_k^{\text{new}}(\Gamma_0(1))$, Serre proved that $\alpha_f=1$.  When $f(z) = \eta^{2}(z) \eta^{2}(11z)$, with
$\eta(z) = q^{1/24} \prod_{n=1}^{\infty} (1-q^{n})$ the usual Dedekind-$\eta$
function, we have that $f(z)\in S_2^{\text{new}}(\Gamma_0(11))$.  In this case, Serre proved that $\alpha_f=\frac{14}{15}$.  He computed the upper bound
\begin{equation}
\label{eqn:serre_density_x011}
\lim_{x\to\infty}\frac{\#\{n\leq x:a_f(n)\neq0\}}{x}<0.847
\end{equation}
and conjectured a lower bound of 0.845.

Using \eqref{Serre-density} and Theorem \ref{count-primes-zero}, we establish a method for computing an explicit lower bound for the density of positive integers $n$ for which $a_f(n)$ is nonzero.  In the cases where $f(z)\in S_{k}^\text{new}(\Gamma_0(1))$ and $f(z)\in S_2^\text{new}(\Gamma_0(11))$, we use this method to prove our third result.

\begin{theorem}
\label{densities}
Suppose that $f(z)=\sum_{n=1}^\infty a_f(n)q^n\in S_{k}^{\mathrm{new}}(\Gamma_0(N))$ is a newform which satisfies Conjecture \ref{automorphy-GRH}.  Define
\[
D_{f}=\lim_{x\to\infty}\frac{\#\{n\leq x:a_f(n)\neq0\}}{x}.
\]
Then the following table provides strict lower bounds for $D_{f}$.
\small
\begin{table}[ht]
\label{density-chart}
\centering
\begin{tabular}{| c | c |}
\hline
$f$ & Strict Lower Bound for $D_{f}$ \\ [0.5ex]
\hline
$\eta^{2}(z) \eta^{2}(11z)$ & $0.8306$\\
\hline
$\Delta_{12}(z)$ & $1-1.54\times10^{-13}$\\
\hline
$\Delta_{16}(z)$ & $1-5.04\times10^{-16}$\\
\hline
$\Delta_{18}(z)$ & $1-5.92\times10^{-17}$\\
\hline
$\Delta_{20}(z)$ & $1-1.35\times10^{-17}$\\
\hline
$\Delta_{22}(z)$ & $1-1.35\times10^{-17}$\\
\hline
$\Delta_{26}(z)$ & $1-3.97\times10^{-18}$\\
\hline
\end{tabular}
\end{table}
\end{theorem}

The proof of Theorem \ref{densities} relies on the fact if $\tau_k(n)=0$, then $n$ must satisfy specific congruences modulo powers of small primes (see Theorem \ref{level1cong}).  The congruences arise from the Galois representations of $\Delta_k(z)$ modulo these primes.  The congruences help us raise the lower bounds on the densities in Theorem \ref{densities}, but the congruences are also of independent interest.

Theorem \ref{count-primes-zero} and Theorem \ref{densities} have useful applications in the study of positive-definite, integer-valued quadratic forms.  Let $Q:\mathbb{Z}^{4r}\to\mathbb{Z}$ be a positive definite, integer-valued quadratic form in $4r$ variables, where $Q(\vec{x})=\frac{1}{2}\vec{x}^T A\vec{x}$ and $\det(A)$ is a square.  Let $N$ be the smallest positive integer such that $NA^{-1}$ has integer entries and even diagonal entries.  If $r_Q(n)=\#\{\vec{x}\in\mathbb{Z}^{4r}:Q(\vec{x})=n\}$, then
\[
\Theta_Q(z)=\sum_{n=0}^{\infty}r_Q(n)q^n\in M_{2r}(\Gamma_0(N)),
\]
where $M_{k}(\Gamma_0(N))$ denotes the space of modular forms of weight $k$ and level $N$.  We can decompose $\Theta_Q(z)$ into the sum of an Eisenstein series
\[
E_Q(z)=\sum_{n=0}^{\infty}a_{\textup{Eis}}(n)q^n
\]
and a cusp form
\[
C_Q(z)=\sum_{n=1}^{\infty}a_{\textup{Cusp}}(n)q^n,
\]
the latter of which can be expressed as a linear combination of newforms in $S_{2r}^{\textup{new}}(\Gamma_0(N))$ (and the images of newforms under the operator $V(d)$ when $d|N$).  Thus we have the identity $r_Q(n)=a_{\textup{Eis}}(n)+a_{\textup{Cusp}}(n)$, where $a_{\textup{Cusp}}(n)$ is approximately $a_{\textup{Eis}}(n)^{1/2+\epsilon}$ for any $\epsilon>0$.  Theorem \ref{count-primes-zero} and Theorem \ref{densities} can be used to analyze the accuracy with which $a_{\textup{Eis}}(n)$ approximates $r_Q(n)$ in the event that $C_{Q}(z)$ is a linear combination of newforms with squarefree level and their images under $V(d)$.  We analyze two examples.

First, consider the quadratic form
\[
Q_1(x,y,z,w)=x^2+y^2+3z^2+3w^2+xz+yw.
\]
It is straightforward to verify that $\Theta_{Q_1}(z)\in M_2(\Gamma_0(11))$.  Furthermore, we can decompose $\Theta_{Q_1}(z)$ as the sum of the Eisenstein series
\[
E_{Q_1}(z)=1+\frac{12}{5}\sum_{n=1}^{\infty}(\sigma_1(n)-11\sigma_1(n/11))q^n
\]
and the cusp form $\frac{8}{5}C_{Q_1}(z)$, where $C_{Q_1}(z) = \eta^{2}(z) \eta^{2}(11z)$. Theorem \ref{count-primes-zero} bounds the number of primes $p$ for which $r_{Q_1}(p)=\frac{12}{5}\sigma_1(p)=\frac{12}{5}(p+1)$.  Furthermore, by Serre's calculation \eqref{eqn:serre_density_x011}, the density of positive integers $n$ for which $r_{Q_1}(n)$ equals its Eisenstein approximation, that is to say
\[
r_{Q_1}(n)=\frac{12}{5}(\sigma_1(n)-11\sigma_1(n/11)),
\]
is at least $0.153$.  By Theorem \ref{densities}, this density is at most $0.1693$.

Second, consider the quadratic form
\begin{equation}
\label{eqn:Q2}
Q_2(x_1,x_2,\ldots,x_{24})=\sum_{k=1}^{24}x_k^{2}.
\end{equation}
It is straightforward to verify that $\Theta_{Q_2}(z)\in M_{12}(\Gamma_0(4))$.  As stated in \cite{BMazur}, we can represent $\Theta_{Q_2}(z)$ as the sum of the Eisenstein series
\[
E_{Q_2}(z)=1+\frac{16}{691}\sum_{n=1}^{\infty}(\sigma_{11}(n)-2\sigma_{11}(n/2)+4096\sigma_{11}(n/4))q^n
\]
and the cusp form
\[
C_{Q_2}(z)=\frac{128}{691}\sum_{n=1}^{\infty}(259\tau_{12}(n)+11920\tau_{12}(n/2)+1060864\tau_{12}(n/4))q^n.
\]
Theorem \ref{count-primes-zero} bounds the number of primes $p$ for which $r_{Q_2}(p)=\frac{16}{691}\sigma_{11}(p)=\frac{16}{691}(p^{11}+1)$.  However, it follows immediately that if Conjecture \ref{Lehmer} is true when $k=12$, then there are no such primes $p$.  Our final result extends this fact to composite $n$.

\begin{theorem}
\label{tau-QF}
Let $Q_2(x_1,x_2,\ldots,x_{24})$ be given by \eqref{eqn:Q2}, and let $n$ be a positive integer.  We have that $r_{Q_2}(n)$ equals its Eisenstein approximation, that is to say
\[
r_{Q_2}(n)=\frac{16}{691}(\sigma_{11}(n)-2\sigma_{11}(n/2)+4096\sigma_{11}(n/4)),
\]
if and only if $\tau_{12}(n)=0$.
\end{theorem}
\noindent
Theorem \ref{densities} then tells us that the density of positive integers $n$ for which $r_{Q_2}(n)$ equals its Eisenstein approximation is less than $1.54\times10^{-13}$.

In Section \ref{sec:background}, we provide the requisite background on symmetric power $L$-functions and Galois representations.  In Section \ref{sec:proofs_of_main_theorems}, we prove Theorems \ref{main-theorem} and \ref{count-primes-zero} assuming Propositions \ref{key-estimates-ab} and \ref{key-estimates-zero}, respectively.  Propositions \ref{key-estimates-ab} and \ref{key-estimates-zero} are proven in Sections \ref{sec:mellin}-\ref{proof-ab}.  In Section \ref{sec:lehmer_densities}, we prove the aforementioned congruences as well as Theorems \ref{densities} and \ref{tau-QF}.

\subsection*{Acknowledgements}

The authors thank Daniel Fiorilli, Ken Ono, Professor Jean-Pierre
Serre, Kannan Soundararajan, Gabor Wiese, and the anonymous referee
for helpful comments. The authors used Magma, Mathematica, and
PARI/GP for computations.

\section{Background}
\label{sec:background}

\subsection{Symmetric power $L$-functions}

Define the numbers $\Lambda_{ \mathrm{Sym}^n f }(j)$ by
\[
-\frac{L'}{L}(s, \mathrm{Sym}^n f )=\sum_{j=1}^{\infty}\frac{\Lambda_{ \mathrm{Sym}^n f }(j)}{j^s},\qquad \re(s)>1.
\]
Let $U_n(x)$ be the $n$-th Chebyshev polynomial of the second type.  Assuming Conjecture \ref{automorphy-GRH}, a straightforward computation shows that for any integer $j$, we have that
\begin{equation}
\label{von-Mangoldt-def}
\Lambda_{ \mathrm{Sym}^n f }(j)=
\begin{cases}
U_n(\cos(m\theta_p))\log p  &\mbox{if $j=p^m$ for some $p\nmid N$ and $m\geq1$,} \\
t_{m,n,p}p^{-mn/2}\log p & \mbox{if $j=p^m$ for some $p\mid N$ and $m\geq1$,}\\
0&\mbox{otherwise,}
\end{cases}
\end{equation}
where $|t_{m,n,p}|=1$.  If $\Lambda(j)$ is the von Mangoldt function, then $|\Lambda_{ \mathrm{Sym}^n f }(j)|\leq j^{-n/2}\Lambda(j)$ whenever $\gcd(j,N)>1$.  Additionally, when $\gcd(j,N)=1$, we have that $|\Lambda_{ \mathrm{Sym}^n f }(j)|\leq(n+1)\Lambda(j)$.  Thus under Conjecture \ref{automorphy-GRH}, we have that $|\Lambda_{ \mathrm{Sym}^n f }(j)|\leq(n+1)\Lambda(j)$ for all $j$, and so $L(s,\mathrm{Sym}^n f)$ satisfies the Ramanujan-Petersson Conjecture.

For future convenience, we record the following identity, which is immediate from Conjecture \ref{automorphy-GRH}, part (b):
{\small
\begin{equation}
\label{gamma-factor}
\frac{\gamma'}{\gamma}(s,\mathrm{Sym}^n f)=\begin{cases}
\displaystyle-\frac{(n+1)\log(2\pi)}{2}+\sum_{j=1}^{(n+1)/2}\frac{\Gamma'}{\Gamma}(s+(j-1/2)(k-1))&\mbox{if $n$ is odd},\\
\displaystyle-\frac{(n+1)\log(2\pi)}{2}+\frac{1}{2}\frac{\Gamma'}{\Gamma}((s+r)/2)+\sum_{j=1}^{n/2}\frac{\Gamma'}{\Gamma}(s+j(k-1))&\mbox{if $n$ is even}.
\end{cases}
\end{equation}}%

Under Conjecture \ref{automorphy-GRH}, the function
\[
\Lambda(s,\mathrm{Sym}^n f)=N^{ns/2}\gamma(s,\mathrm{Sym}^n f)L(s,\mathrm{Sym}^n f)
\]
is an entire function of order 1 for all $n\geq1$.  By the Hadamard factorization theorem, there exist constants $a_{\mathrm{Sym}^n f}$ and $b_{\mathrm{Sym}^n f}$ such that
\[
\Lambda(s,\mathrm{Sym}^n f)=e^{a_{\mathrm{Sym}^n f}+b_{\mathrm{Sym}^n f}s}\prod_{\rho\neq0,1}\left(1-\frac{s}{\rho}\right)e^{s/\rho},
\]
where $\rho$ ranges over the zeros of $\Lambda(s,\mathrm{Sym}^n f)$ different from 0 and 1.  Taking the logarithmic derivative of both sides, we obtain the identity
\begin{equation}
\label{hadamard}
-\frac{L'}{L}(s,\mathrm{Sym}^n f)=\frac{n}{2}\log N+\frac{\gamma'}{\gamma}(s,\mathrm{Sym}^n f)-b_{\mathrm{Sym}^n f}-\sum_{\rho\neq0,1}\left(\frac{1}{s-\rho}+\frac{1}{\rho}\right).
\end{equation}
By \cite[Proposition 5.7]{IK}, we have that
\begin{equation}
\label{bsym}
\re(b_{\mathrm{Sym}^n f})=-\sum_{\rho}\re\left(\frac{1}{\rho}\right).
\end{equation}

\subsection{Galois representations and congruences}

We will require the theory of Galois representations attached to $f$ to show that there are few small primes $p$ for which $a_f(p) = 0$. Suppose that $f$ has weight $k$, level $N$, and the Fourier coefficients of $f$ are rational integers.  For every prime $\ell$ there is a representation $\rho : {\rm Gal}(\overline{\mathbb{Q}}/\mathbb{Q}) \to \mathrm{GL}_{2}(\mathbb{F}_{\ell})$ with the property that the fixed field of $\ker \rho$ is ramified only at primes $p$ dividing $N\ell$.  If $p \nmid N\ell$, then
\[
  {\rm tr}~\rho({\rm Frob}_{p}) \equiv a_f(p) \pmod{\ell}, \text{ and }
  {\rm det}~\rho({\rm Frob}_{p}) \equiv p^{k-1} \pmod{\ell}.
\]
Let $\tilde{\rho}$ be the composition ${\rm Gal}(\overline{\mathbb{Q}}/\mathbb{Q})
\to \mathrm{GL}_{2}(\mathbb{F}_{\ell}) \to \mathrm{PGL}_{2}(\mathbb{F}_{\ell})$. In \cite[Chapter 7]{Bosman}, Johan Bosman explicitly computes the fixed fields of $\ker \tilde{\rho}$
for level $1$ newforms of weights $12$ through $22$ and $\ell \leq 23$. These
are given as polynomials of degree $\ell+1$ whose splitting field is the corresponding extension (with Galois group a subgroup of $\mathrm{PGL}_{2}(\mathbb{F}_{\ell})$). These
computations were extended by Mascot \cite{Mascot}, who computed polynomials
that allow recovery of $\tau(p) \bmod \ell$ for primes $\ell$ up to $31$.

A straightforward calculation \cite[Lemma 7.5.1]{Bosman} shows that
we have ${\rm tr}~\rho(M) \equiv 0 \pmod{\ell}$ if and only if
$\tilde{\rho}(M)$ has order $2$. It follows that we may determine whether
or not $a_f(p) \equiv 0 \pmod{\ell}$ from the factorization of a polynomial
defining the fixed field of $\ker \tilde{\rho}$ in $\mathbb{F}_{p}[x]$.

More information can be obtained when the representations are reducible. We
will use the following results to prove congruences for level $1$ newforms
modulo powers of small primes. As usual, for an even integer $k$ let
\[
  E_{k}(z) = 1 - \frac{2k}{B_{k}} \sum_{n=1}^{\infty} \sigma_{k-1}(n) q^{n}
\]
denote the usual weight $k$ Eisenstein series. For $k \geq 4$ we have
$E_{k}(z) \in M_{k}(\Gamma_{0}(1))$.

\begin{lemma}
\label{Eisensteinlem}
If $p$ is a prime and $r \geq 1$, then $E_{p^{r-1} (p-1)} \equiv 1 \pmod{p^{r}}$.
\end{lemma}
\begin{proof}
This follows immediately from the formula for $B_{k}$ given in
\cite[Theorem 3, pg. 233]{IrelandRosen}.\end{proof}

Define Ramanujan's $\theta$-operator by $\theta(f) = \frac{1}{2 \pi i} \frac{d}{dz} f(z)$, and define $E_{2,p^r}(z)=E_2(z)-p^r E_2(p^r z)$. For a prime $p$ and a positive integer $n$, let
${\rm ord}_{p}(n)$ be the highest power of $p$ that divides $n$.

\begin{lemma}
\label{Primepower}
If $f \in M_{k}(\Gamma_{0}(N))$ and $p^{r} | N$, then there is a form
$g \in M_{k+2}(\Gamma_{0}(N))$ with
\[
  \theta(f) \equiv g \pmod{p^{r - {\rm ord}_{p}(12)}}.
\]
\end{lemma}
\begin{proof}
Proposition 2.11 of \cite{OnoBook} states that $h = \theta(f) - \frac{k E_{2} f}{12} \in M_{k+2}(\Gamma_{0}(N))$. Now, let
\[
  g(z) = h(z) + \frac{k E_{2,p^{r}}(z) f(z)}{12}.
\]
Then $g \equiv \theta(f) \pmod{p^{r - {\rm ord}_{p}(12)}}$ and $g(z) \in
M_{k+2}(\Gamma_{0}(N))$.
\end{proof}

The Kummer congruences \cite[Theorem 5, pg. 239]{IrelandRosen} imply
that $E_{2} \equiv E_{p+1} \pmod{p}$; combining this with
\cite[Proposition 2.11]{OnoBook}, we see that if $f \in M_{k}(\Gamma_{0}(1))$, then
there is a $g \in M_{k+p+1}(\Gamma_{0}(1))$ with $\theta(f) \equiv g \pmod{p}$.
Finally, to prove congruences between two different modular forms, we will
use the following well-known theorem of Sturm \cite{Sturm}.
\begin{theorem}
\label{Sturmtheorem}
Let $f(z) = \sum_{n=0}^{\infty} \in M_{k}(\Gamma_{0}(N))$ with $a_f(n) \in \mathbb{Z}$ for all $n$. If $a_f(n) \equiv 0 \pmod{M}$ for $n \leq \frac{k}{12} [{\rm SL}_{2}(\mathbb{Z}) : \Gamma_{0}(N)]$, then $a_f(n) \equiv 0 \pmod{M}$ for all $n \geq 0$.
\end{theorem}
To isolate specific residue classes, we will use \cite[Proposition 2.8]{OnoBook}.
\begin{theorem}
\label{twisting}
Suppose that $f(z) = \sum_{n=0}^{\infty} a_f(n) q^{n} \in M_{k}(\Gamma_{0}(N))$
and $\psi$ is a quadratic Dirichlet character with modulus $m$. Then
\[
  \sum_{n=0}^{\infty} a_f(n) \psi(n) q^{n} \in M_{k}(\Gamma_{0}(Nm^{2})).
\]
\end{theorem}
Finally, define the $U(d)$ and $V(d)$ operators by
\begin{align*}
  \left( \sum_{n=0}^{\infty} a_f(n) q^{n} \right) | U(d) &=
  \sum_{n=0}^{\infty} a(dn) q^{n}\\
  \left( \sum_{n=0}^{\infty} a_f(n) q^{n} \right) | V(d) &=
  \sum_{n=0}^{\infty} a_f(n) q^{dn}.
\end{align*}
Proposition 2.22 of \cite{OnoBook} implies that if
$f(z) \in M_{k}(\Gamma_{0}(N))$, then
$f(z) | V(d) \in M_{k}(\Gamma_{0}(dN))$. It follows from \cite[Section 4.1.3]{Kil}
that if $f(z) \in M_{k}(\Gamma_{0}(N))$ and $p$ is a prime divisor
of $N$, then $f(z) | U(p) \in M_{k}(\Gamma_{0}(N))$. Moreover,
if $p^{2} | N$, then $f(z) | U(p) \in M_{k}(\Gamma_{0}(N/p))$ \cite[Lemma 7]{atkin-lehner}.

\section{Proofs of Theorems \ref{main-theorem} and \ref{count-primes-zero}}
\label{sec:proofs_of_main_theorems}

To estimate $\pi_{f,I}(x)$, we use the following approximation for the indicator function of an interval.

\begin{lemma}
\label{BS}
Let $I=[\alpha,\beta]\subset[0,\pi]$ be a subinterval, and let $M$ be a positive integer.  There exist trigonometric polynomials
\[
F_{I,M}^{\pm}(\theta)=\sum_{n=0}^M\hat{F}_{I,M}^{\pm}(n)U_n(\cos \theta)
\]
that satisfy the following properties:
\begin{enumerate}
\item For all $0\leq\theta\leq\pi$, we have
\[
F_{I,M}^{-}(\theta)\leq\chi_I(\theta)\leq F_{I,M}^{+}(\theta).
\]
\item We have
\[
|\hat{F}_{I,M}^{\pm}(0)-\mu_{ST}(I)|\leq \frac{4}{M+1}.
\]
\item For all $1\leq n\leq M$, we have
\[
|\hat{F}_{I,M}^{\pm}(n)|\leq4\left(\frac{1}{M+1}+\min\left\{\frac{\beta-\alpha}{2\pi},\frac{1}{\pi n}\right\}\right).
\]
\end{enumerate}
\end{lemma}
\begin{proof}
We begin with some notation.  Let $e(t)=e^{2\pi i t}$, and let $J=[\frac{\alpha}{2\pi},\frac{\beta}{2\pi}]$.  The Fourier expansion of the indicator function $\chi_J(\theta)$ is given by
\[
\chi_J(\theta)=\sum_{n\in\mathbb{Z}}\hat{\chi}_J(n)e(n\theta),\qquad \hat{\chi}_J(n)=\int_{J}e(-nt)dt.
\]
We note that $\hat{\chi}_J(0)=\frac{\beta-\alpha}{2\pi},$
\[
|\hat{\chi}_J(n)|\leq\min\left\{\frac{\beta-\alpha}{2\pi},\frac{1}{\pi|n|}\right\}
\]
when $n\neq0$, and
\[
\mu_{ST}(I)=2\re(\hat{\chi}_J(0)-\hat{\chi}_J(2)).
\]

Let $M$ be a positive integer.  In \cite[Chapter 1, Section 2]{Montgomery}, it is proven that there exist trigonometric polynomials
\[
S_{J,M}^{\pm}(\theta)=\sum_{|n|\leq M}\hat{S}_{J,M}^{\pm}(n)e(n\theta)
\]
that satisfy the following three properties:
\begin{enumerate}
\item For all $\theta\in[0,1]$, we have $S_{J,M}^{-}(\theta)\leq\chi_{J}(\theta)\leq S_{J,M}^{+}(\theta).$
\item For all $n$, we have $S_{J,M}^{\pm}(n)+S_{J,M}^{\pm}(-n)=2\re(S_{J,M}^{\pm}(n))$.
\item For all $0\leq|n|\leq M$, we have
\[
|\hat{S}_{J,M}^{\pm}(n)-\hat{\chi}_{J}(n)|\leq \frac{1}{M+1}.
\]
\end{enumerate}
Now, define
\[
F_{I,M}^{\pm}(\theta)=S_{J,M}^{\pm}\left(\frac{\theta}{2\pi}\right)+S_{J,M}^{\pm}\left(-\frac{\theta}{2\pi}\right).
\]
We want express $F_{I,M}^{\pm}(\theta)$ in terms of the orthonormal basis for $L^2([0,\pi],\mu_{ST})$ given by Chebyshev polynomials of the second kind $\{U_n(\cos \theta)\}_{n=0}^\infty$.  It is well known that
\[
U_n(\cos\theta)=\frac{\sin((n+1)\theta)}{\sin\theta}.
\]
We have $1=U_0(\cos \theta),$ $\cos \theta=\frac{1}{2}U_1(\cos \theta)$, and
\[
\cos(n\theta)=\frac{1}{2}(U_n(\cos \theta)-U_{n-2}(\cos \theta)),\qquad n\geq2.
\]
Therefore, we have
\[
F_{I,M}^{\pm}(\theta)=2\sum_{n=0}^M \re(\hat{S}_{J,M}^{\pm}(n)-\hat{S}_{J,M}^{\pm}(n+2))U_n(\cos \theta),
\]
where $\hat{S}_{J,M}^{\pm}(n)=0$ if $n>M$.  The lemma now follows.
\end{proof}

Assuming the Riemann Hypothesis, Schoenfeld \cite{SchoenfeldGRH} proved that if $x\geq 2657$, then
\[
|\pi(x)-\mathrm{Li}(x)|\leq\frac{1}{8\pi}\sqrt{x}\log x.
\]
To bound $\pi_{f,I}(x)$, we use the following variant of the Erd\H{o}s-Tur\'an inequality \cite[Chapter 1, Theorem 1]{Montgomery}, which is a straightforward consequence of Lemma \ref{BS} and Schoenfeld's inequality.

\begin{lemma}
\label{ET}
Let $I=[\alpha,\beta]\subset[0,\pi]$ be a subinterval, and let $M\geq2$.  If the Riemann Hypothesis is true, then for all $x\geq 2657$, we have
\begin{align*}
|\pi_{f,I}(x)-\mu_{ST}(I)\mathrm{Li}(x)|&\leq\frac{4}{M+1}\left(\mathrm{Li}(x)+\frac{1}{8\pi}\sqrt{x}\log x\right)+\frac{1}{8\pi}\sqrt{x}\log x\\
&+4\sum_{n\leq M}\left(\frac{1}{M+1}+\frac{1}{\pi n}\right)\left|\sum_{p\leq x}U_n(\cos \theta_p )\right|.
\end{align*}
\end{lemma}

Theorem \ref{main-theorem} now follows from the next proposition, whose proof we defer to Section \ref{proof-ab}.  Recall that $\mathfrak{q}(f)=N(k-1)$.
\begin{proposition}
\label{key-estimates-ab}
Assume the hypotheses of Theorem \ref{main-theorem}.  If $n\geq1$ and $x\geq5\times10^5$, then
\begin{align*}
\left|\sum_{p\leq x}U_n(\cos \theta_p )\right|&\leq n\log n\left(\frac{5\sqrt{x}}{12}+\frac{7\sqrt{x}}{2\log x}\right)+n\left(\frac{8}{75}\sqrt{x}\log x + 7\sqrt{x}\log\mathfrak{q}(f)\right)\\
&+15\sqrt{x}\log n+4(\log\mathfrak{q}(f))\sqrt{x}\log x.
\end{align*}
\end{proposition}

Assuming the truth of Proposition \ref{key-estimates-ab}, we prove Theorem \ref{main-theorem}.
\begin{proof}[Proof of Theorem \ref{main-theorem}]
Choose $M=2.40753x^{1/4}(\log x)^{-1}$.  Since $M\geq2$ when $x\geq 5\times 10^5$, Lemma \ref{ET} holds.  With this choice of $M$, we use Proposition \ref{key-estimates-ab} to estimate the right hand side of Lemma \ref{ET} and obtain the claimed result when $x\geq 5\times 10^5$.  Noting that the error term is greater than $\pi(x)$ for $2\leq x<5\times10^5$, Theorem \ref{main-theorem} now holds for all $x\geq 2$.
\end{proof}

To prove Theorem \ref{count-primes-zero}, we sharpen our estimates by weighing the contribution from each prime by a test function which is a pointwise upper bound for the indicator function of the interval $[x,2x]$.  If we let
\[
g(y)=\begin{cases}
\displaystyle\exp\left(\frac{4}{3}+\frac{1}{(y-\frac{1}{2})(y-\frac{5}{2})}\right)&\mbox{if $\displaystyle\frac{1}{2}<y<\frac{5}{2}$,}\\
0&\mbox{otherwise}
\end{cases}
\]
and $g_x(y)=g(y/x)$, then $g_x(y)$ is such a test function.

For $M\geq 2$, define $I_M=[\frac{\pi}{2}(1-\frac{1}{2M}),\frac{\pi}{2}(1+\frac{1}{2M})]$; for this choice of interval, we have $\mu_{ST}(I_M)\leq\frac{1}{M}$.  We use the upper bound for $\chi_I(\theta)$ given by $F_{I,M}^+(\theta)$ in Lemma \ref{BS}.  By our choice of $I$, $F_{I,M}^{+}(\theta)$ is symmetric across the vertical line $\theta=\frac{\pi}{2}$.  Thus if $n$ is odd, then $\hat{F}_{I,M}^{+}(n)$ vanishes.  Using the bound $1\leq\frac{\log p}{\log x}$ for all $x<p\leq 2x$, we immediately obtain the following lemma.
\begin{lemma}
\label{BS-upper}
If $M$ is a positive integer and $x\geq2$, then
\[
\pi_{f,I_M}(2x)-\pi_{f,I_M}(x)\leq \frac{5}{M\log x}\sum_{n\leq M/2}\left|\sum_{p}U_{2n}(\cos \theta_p )g_x(p)\log p\right|.
\]
\end{lemma}

Theorem \ref{count-primes-zero} now follows from the next proposition, whose proof we defer to Section \ref{proof-ab}.
\begin{proposition}
\label{key-estimates-zero}
Assume the above notation.  Assume that the symmetric power $L$-functions of $f$ are automorphic and satisfy the Generalized Riemann Hypothesis.  If $n\geq0$ and $x\geq1.4\times10^7$, then
\[
\left|\sum_{p}U_n(\cos \theta_p )g_x(p)\log p\right|
\]
is bounded above by
\begin{align*}
1.6844 \delta_{n,0}x+\sqrt{x}&[8.1736n\log n+(21.038+8.188\log\mathfrak{q}(f))n\\
&+57.22\log n+134.01+8.3\log\mathfrak{q}(f)],
\end{align*}
where $\delta_{n,0}=1$ if $n=0$ and $0$ otherwise.
\end{proposition}

Assuming the truth of Proposition \ref{key-estimates-zero}, we prove Theorem \ref{count-primes-zero}.
\begin{proof}[Proof of Theorem \ref{count-primes-zero}]
Choose $M=1.8159x^{1/4}(\log x)^{-1/2}$.  Since $M\geq2$ when $x\geq 1.4\times10^7$, Lemma \ref{ET} holds.  With this choice of $M$, we use Proposition \ref{key-estimates-zero} to estimate the right hand side of Lemma \ref{BS-upper} and obtain the claimed result when $x\geq 1.4\times10^7$.  Noting that the error term is greater than $\pi(x)$ for $2\leq x<1.4\times10^7$, Theorem \ref{count-primes-zero} now holds for all $x\geq 2$.
\end{proof}

\section{The Mellin Transform}
\label{sec:mellin}

Let $ L(s,\mathrm{Sym}^n f) $ be the $n$-th symmetric power $L$-function associated to a newform $f$ satisfying Conjecture \ref{automorphy-GRH}.  As shown in the previous section, the conclusion of Theorem \ref{main-theorem} will follow directly from Proposition \ref{key-estimates-ab}, which is a completely explicit version of the prime number theorem for $L(s, \mathrm{Sym}^n f )$ for each $n\geq1$.  Define
\begin{equation}
\label{psi-def}
\psi_{ \mathrm{Sym}^n f }(x)=\sum_{j\leq x}\Lambda_{ \mathrm{Sym}^n f }(j).
\end{equation}
Let $\sigma_0=1+\frac{1}{\log(x)}$.  A standard application of Mellin inversion gives us
\[
\psi_{ \mathrm{Sym}^n f }(x)=-\frac{1}{2\pi i}\int_{\sigma_0-i\infty}^{\sigma_0+i\infty}\frac{L'}{L}(s, \mathrm{Sym}^n f )\frac{x^s}{s}~ds+\frac{1}{2}\Lambda_{ \mathrm{Sym}^n f }(x).
\]
Since the above integral does not converge absolutely, we approximate $\psi_{ \mathrm{Sym}^n f }(x)$ with a truncated integral.
\begin{lemma}
\label{psi-estimate}
If $2\leq T\leq x$, then
\begin{align*}
&\left|\psi_{ \mathrm{Sym}^n f }(x)-\left(-\frac{1}{2\pi i}\int_{\sigma_0-iT}^{\sigma_0+iT}\frac{L'}{L}(s, \mathrm{Sym}^n f )\frac{x^s}{s}~ds\right)\right|\\
\leq& (n+1)\left(\frac{6.343 x(\log x)^2}{T}+\frac{12.24 x\log x}{T}+\frac{\log x}{T}+7.8\log x\right)
\end{align*}
\end{lemma}
\begin{proof}
Since $|\Lambda_{ \mathrm{Sym}^n f }(j)|\leq(n+1)\Lambda(j)$, it follows from the arguments in \cite[Chapter 17]{Davenport} that the quantity in the statement of the lemma is bounded by
\begin{align*}
\sum_{j=1}^\infty&|\Lambda_{ \mathrm{Sym}^n f }(j)|(x/j)^{\sigma_0}\min\{1,|T\log(x/j)|^{-1}\}+\sigma_0 T^{-1}\Lambda_{ \mathrm{Sym}^n f }(x)\\
&\leq (n+1)\left(ex\sum_{j=1}^\infty\Lambda(j)j^{-\sigma_0}\min\{1,|T\log(x/j)|^{-1}\}+T^{-1}\log(x)+1\right).
\end{align*}
If $j\notin(\frac{3x}{4},\frac{5x}{4})$, then $|\log(x/j)|^{-1}\leq\frac{9}{2},$ so
\[
\sum_{j\notin(\frac{3x}{4},\frac{5x}{4})}\Lambda(j)j^{-\sigma_0}\min\{1,|T\log(x/j)|^{-1}\}\leq\frac{9}{2T}\sum_{j\notin(\frac{3x}{4},\frac{5x}{4})}\Lambda(j)j^{-\sigma_0}\leq\frac{9\log(x)}{2T}.
\]
Let $x_1$ be the largest prime power in the range $(\frac{3x}{4},x)$.  If $j=x_1$, then $\frac{1}{|\log(x/j)|}\leq\frac{x}{x-x_1}.$  Thus the contribution arising from $x_1$ is $\Lambda(x_1)x_1^{-\sigma_0}\leq\frac{4\log(x)}{3x}.$

If $j\in(\frac{3x}{4},x_1)$, then $\frac{1}{|\log(x/j)|}\leq\frac{x_1}{x_1-j}.$ Thus the contribution from these values of $j$ is bounded by
{\small\[
\sum_{j\in(\frac{3x}{4},x_1)}\Lambda(j)j^{-\sigma_0}\frac{x_1}{T(x_1-j)}\leq \frac{x_1}{T}\sum_{j\in(0,\frac{x}{4})}\frac{\Lambda(x_1-j)}{(x_1-j)^{\sigma_0} j}\leq \frac{x_1}{T}\sum_{j\in(0,\frac{x}{4})}\frac{\log(x+j)}{(\frac{3x}{4}-j)j}\leq\frac{4\log(x)^2}{3T}.
\]}%

Let $x_2$ be the least prime power greater than $x$.  If $j=x_2$, then $\frac{1}{|\log(x/x_2)|}\leq\frac{x_2}{x_2-x}.$ Thus the contribution arising from $x_2$ is bounded by $\Lambda(x_2)x_2^{-\sigma_0}\leq\frac{\log(x)}{x}.$

If $j\in(x_2,\frac{5x}{4})$, then $\frac{1}{|\log(x/j)|}\leq\frac{j}{j-x_2}.$  Thus the contribution from these values of $j$ is bounded by
\[
\sum_{j\in(x_2,\frac{5x}{4})}\Lambda(j)j^{-\sigma_0}\frac{j}{T(j-x_2)}\leq\frac{1}{T}\sum_{j\in(0,\frac{x}{4})}\frac{\Lambda(x_2+j)(x_2+j)}{(x_2+j)^{\sigma_0}j}\leq \frac{(\log x)^2}{T}.
\]
Collecting the above estimates yields the desired result.
\end{proof}

\section{The Density of Zeros of Symmetric Power $L$-functions}
\label{sec:density_of_zeros}

In order to estimate the integrals in the previous section, we must understand the distribution of the poles of $\frac{L'}{L}(s, \mathrm{Sym}^n f )$.  We first present some auxiliary results.  Throughout this section, we have $n\geq1$.
\begin{lemma}
\label{aux1}
If $s=\sigma+it$ and $\sigma>1$, then
\[
\left|\frac{L'}{L}(s, \mathrm{Sym}^n f )\right|\leq-(n+1)\frac{\zeta'}{\zeta}(\sigma).
\]
\end{lemma}
\begin{proof}
Since $|\Lambda_{ \mathrm{Sym}^n f }(j)|\leq(n+1)\Lambda(j)$, we have
\[
\left|-\frac{L'}{L}(s, \mathrm{Sym}^n f )\right|\leq\sum_{j=1}^\infty\left|\frac{\Lambda_{ \mathrm{Sym}^n f }(j)}{j^s}\right|\leq(n+1)\sum_{j=1}^\infty\frac{\Lambda(j)}{j^{\sigma}}\leq-(n+1)\frac{\zeta'}{\zeta}(\sigma).
\]
\end{proof}
\begin{lemma}
\label{digamma}
Let $\left<s\right>=\min\{|s+j|:j\in\mathbb{Z}\cap[0,\infty)\}$.  For all $s\in\mathbb{C}$, we have
\[
\left|\frac{\Gamma'}{\Gamma}(s)\right|\leq4+\frac{2}{\left<s\right>}+2\log(|s|+3).
\]
If $\re(s)>0$, then
\[
\re\left(\frac{\Gamma'}{\Gamma}(s)\right)\leq \frac{1}{\re(s)}+\log|s|.
\]
\end{lemma}
\begin{proof}
The second estimate follows from the proof of Lemma 4 of \cite{Ono}, which also shows that if $\re(s)\geq1$, then
\[
\left|\frac{\Gamma'}{\Gamma}(s)\right|\leq \frac{11}{3}+\log(|s|+1).
\]
We may now suppose that $\re(s)<1$.  For any integer $m\geq1$, we have the  identity
\[
\frac{\Gamma'}{\Gamma}(s)=\frac{\Gamma'}{\Gamma}(s+m)-\sum_{j=0}^{m-1}\frac{1}{s+j}=\frac{\Gamma'}{\Gamma}(s+m)-\sum_{j=0}^{m-1}\frac{1}{s+m-1-j}.
\]
We choose $m=2+\lfloor-\re(s)\rfloor$.  We then have $1\leq\re(s+m)\leq2$, so
\[
\left|\frac{\Gamma'}{\Gamma}(s+m)\right|\leq\frac{11}{3}+\log(|\re(s+m)+i\im(s)|+1)\leq \frac{11}{3}+\log(|\im(s)|+3).
\]
By our choice of $m$, we now have
\begin{align*}
\left|\frac{\Gamma'}{\Gamma}(s)\right|&\leq\left|\frac{\Gamma'}{\Gamma}(s+m)\right|+\sum_{j=0}^{m-1}\frac{1}{|s+m-1-j|}\\
&\leq\frac{11}{3}+\log(|\im(s)|+3)+\frac{2}{\left<s\right>}+\sum_{j=0}^{m-3}\frac{1}{|s+m-3-j|}\\
&\leq\frac{11}{3}+\log(|\im(s)|+3)+\frac{2}{\left<s\right>}+\sum_{j=0}^{m-3}\frac{1}{j+1}\\
&\leq \frac{7}{2}+\frac{2}{\left<s\right>}+2\log(|s|+3).
\end{align*}
The result now follows.
\end{proof}

\begin{lemma}
\label{log-deriv-gamma-factor}
If $\re(s)>-\frac{1}{2}$ and $|s|\geq\frac{1}{8}$, then
\begin{align*}
\left|\frac{\gamma'}{\gamma}(s, \mathrm{Sym}^n f )\right|&\leq (n+1) (\log (k-1)+\log (n+2 |s|+9)+8.5)\\
&+7 \log (n+2 |s|+9)+\log (|s|+7)+5.
\end{align*}
If $\re(s)\geq2$ and $\im(s)=T$, then
\[
\re\left(\frac{\gamma'}{\gamma}(s, \mathrm{Sym}^n f )\right)\leq \frac{n+1}{2} (\log (k-1)+\log (n+|T|+3)-1)+\frac{7}{2} \log (n+|T|+3).
\]
\end{lemma}
\begin{proof}
Under the above conditions on $s$, we have that $\left<s\right>\geq\frac{1}{8}$.  Thus the claimed results follow from Lemma \ref{digamma}, Stirling's bounds for the gamma function, and the shape of $\gamma(s, \mathrm{Sym}^n f )$.
\end{proof}

We now give the distribution of zeros of $L(s, \mathrm{Sym}^n f )$ in the critical strip.

\begin{lemma}
\label{zero-count}
For any $T$, let
\[
N_{ \mathrm{Sym}^n f }(T)=\#\{\rho=1/2+i\gamma:L(\rho, \mathrm{Sym}^n f )=0,|\gamma- T|\leq1\}.
\]
For a nonnegative integer $j$, let
\[
N_{ \mathrm{Sym}^n f }^*(j)=\#\{\rho=1/2+i\gamma:L(\rho, \mathrm{Sym}^n f )=0,j\leq|\gamma|\leq j+1\}.
\]
We have
\[
N_{ \mathrm{Sym}^n f }(T)\leq\frac{13}{12} (n+1) (\log \mathfrak{q}(f)+\log (n+|T|+3))+8\log (n+|T|+3).
\]
and
\[
N_{ \mathrm{Sym}^n f }^*(j)\leq \frac{5}{6} (n+1) (\log \mathfrak{q}(f)+\log (n+j+7/2))+6\log (n+j+7/2).
\]
\end{lemma}
\begin{proof}
Let $s_0=2+iT$.  By the arguments in \cite{Murty}, we have that
\begin{equation}
\label{eqn:sum_zeros}
\sum_{\rho}\re\left(\frac{1}{s_0-\rho}\right)=\frac{n}{2}\log N+\re\left(\frac{\gamma'}{\gamma}(s_{0}, \mathrm{Sym}^n f )\right)+\re\left(\frac{L'}{L}(s_{0}, \mathrm{Sym}^n f )\right),
\end{equation}
where the sum is over the nontrivial zeroes of $ L(s,\mathrm{Sym}^n f) $.  We use Lemma \ref{aux1} to bound $\re(\frac{L'}{L}(s_{0}, \mathrm{Sym}^n f ))$ and part 2 of Lemma \ref{log-deriv-gamma-factor} to bound the contribution from $\re(\frac{\gamma'}{\gamma}(s_{0}, \mathrm{Sym}^n f ))$.  Collecting these two estimates, we find that
\[
\sum_{\rho}\re\left(\frac{1}{s_0-\rho}\right)\leq \frac{n+1}{2} (\log \mathfrak{q}(f)+\log (n+|T|+3))+\frac{7}{2} \log (n+|T|+3).
\]

We first estimate $N_{ \mathrm{Sym}^n f }(T)$.  If $\rho=\frac{1}{2}+i\gamma$ is a nontrivial zero of $ L(s,\mathrm{Sym}^n f) $ with $|\gamma-T|\leq1$, then
\[
\re\left(\frac{1}{s_0-\rho}\right)\geq\frac{6}{13}.
\]
Since
\[
N_{ \mathrm{Sym}^n f }(T)\leq\frac{13}{6}\sum_{\substack{\rho \\ |\gamma- T|\leq 1}}\re\left(\frac{1}{s_0-\rho}\right)\leq\frac{13}{6}\sum_{\rho}\re\left(\frac{1}{s_0-\rho}\right),
\]
the claimed estimate for $N_{ \mathrm{Sym}^n f }(T)$ follows.

Now we estimate $N_{ \mathrm{Sym}^n f }^*(j)$.  If $T-\frac{1}{2}\leq\gamma\leq T+\frac{1}{2}$, then
\[
\re\left(\frac{1}{s_0-\rho}\right)\geq\frac{3}{5}.
\]
Setting $T=j+\frac{1}{2}$, we have that
\[
N_{ \mathrm{Sym}^n f }^*(j)\leq\frac{5}{3}\sum_{\substack{\rho \\ j\leq|\gamma|\leq j+1}}\re\left(\frac{1}{s-\rho}\right)\leq\frac{5}{3}\sum_{\rho}\re\left(\frac{1}{s-\rho}\right).
\]
The second result now follows.
\end{proof}

We now estimate $\frac{L'}{L}(s, \mathrm{Sym}^n f )$ in certain vertical strips.

\begin{lemma}
\label{aux2}
If $s=\sigma+iT$ with $-1/2\leq\sigma\leq3$ and $|s|\geq1/8$, then
\begin{align*}
\left|\frac{L'}{L}(s, \mathrm{Sym}^n f )-\sum_{|\gamma-T|\leq1}\frac{1}{s-\rho}\right|\leq(n+1) [4.88 &\log \mathfrak{q}(f)+7.35 \log (n+|T|+5)+19.2]\\
&+46.13 \log (n+|T|+5)+40.
\end{align*}
\end{lemma}
\begin{proof}
Logarithmically differentiating the Hadamard product and functional equation for $L(s, \mathrm{Sym}^n f )$, we have
\[
\frac{L'}{L}(s, \mathrm{Sym}^n f )=B_{ \mathrm{Sym}^n f }+\sum_{\rho}\left(\frac{1}{s-\rho}+\frac{1}{\rho}\right)-\frac{n}{2}\log(N)-\frac{\gamma'}{\gamma}(s, \mathrm{Sym}^n f ).
\]
Evaluating this expression at $s=\sigma+iT$ and $3+iT$ and subtracting the resulting equations (in order to eliminate $B_{ \mathrm{Sym}^n f }$), we have
\begin{align*}
\frac{L'}{L}(s, \mathrm{Sym}^n f )-\frac{L'}{L}(3+iT, \mathrm{Sym}^n f )&=-\frac{\gamma'}{\gamma}(s, \mathrm{Sym}^n f )+\frac{\gamma'}{\gamma}(3+iT, \mathrm{Sym}^n f )\\
&+\sum_{\rho}\left(\frac{1}{s-\rho}-\frac{1}{3+iT-\rho}\right).
\end{align*}
Using Lemma \ref{log-deriv-gamma-factor}, we have
\begin{align*}
\left|\frac{L'}{L}(s, \mathrm{Sym}^n f )-\sum_{|\gamma-T|\leq1}\frac{1}{s-\rho}\right|&\leq2(n+1) (\log (k-1)+\log (n+2|T|+15)+8.5)\\
&+8 \log (n+2|T|+15)+31\\
&+\sum_{|\gamma-T|>1}\left|\frac{1}{s-\rho}-\frac{1}{3+iT-\rho}\right|\\
&+\sum_{|\gamma-T|\leq1}\left|\frac{1}{3+iT-\rho}\right|.
\end{align*}

The second sum has $N_{ \mathrm{Sym}^n f }(T)$ terms which each have absolute value at most $\frac{1}{2}$.  Using Lemma \ref{zero-count}, we bound the first sum by
\[
\sum_{|\gamma-T|\geq1}\left|\frac{1}{s-\rho}-\frac{1}{3+iT-\rho}\right|\leq 3\sum_{|\gamma-T|\geq1}\frac{1}{1+|\gamma-T|^2}\leq\frac{87}{20}\sum_{|\gamma-T|\geq1}\frac{5/2}{25/4+|\gamma-T|^2}.
\]
By replacing $2+iT$ with $3+iT$ in \eqref{eqn:sum_zeros}, and using the GRH-dependent equality
\[
\frac{5/2}{25/4+|T-\gamma|^2}=\re\left(\frac{1}{3+iT-\rho}\right),
\]
we find that
\[
\sum_{|\gamma-T|\geq1}\frac{5/2}{25/4+|T-\gamma|^2}\leq \frac{n}{2}\log N+\re\left(\frac{\gamma'}{\gamma}(3+iT, \mathrm{Sym}^n f )\right)+\re\left(\frac{L'}{L}(3+iT, \mathrm{Sym}^n f )\right).
\]
The claimed result follows by collecting the preceding estimates and invoking Lemma \ref{log-deriv-gamma-factor}.
\end{proof}

\section{The Contour Integral}
\label{sec:contour}

Let $n\geq1$, let $T\geq3$ be a number which is not the ordinate of any zero of $L( \mathrm{Sym}^n f ,s)$, and let $U+\frac{1}{4}>0$ be a large integer.  Let $\gamma(T,U)=S_1\cup S_2\cup S_3$, where
\begin{align*}
S_1&=\{-U+it:|t|\leq T\}\\
S_2&=\{\sigma\pm iT:-U\leq\sigma\leq-1/4\},\\
S_3&=\{\sigma\pm iT:-1/4\leq\sigma\leq \sigma_0\}.
\end{align*}
Recall that $\sigma_0=1+\frac{1}{\log x}$.  By the argument principle, if $\rho=\beta+i\gamma$ with $\beta,\gamma\in\mathbb{R}$, we have
\begin{align*}
&-\frac{1}{2\pi i}\int_{\sigma_0-iT}^{\sigma_0+iT}\frac{L'}{L}(s, \mathrm{Sym}^n f )\frac{x^s}{s}~ds\\
&=-\frac{1}{2\pi i}\int_{\gamma(T,U)}\frac{L'}{L}(s, \mathrm{Sym}^n f )\frac{x^s}{s}~ds-\sum_{\substack{L(\rho, \mathrm{Sym}^n f )=0 \\ |\gamma|\leq T \\ -U\leq\beta<1 \\ \rho\neq0}}\frac{x^\rho}{\rho}-\textup{Res}_{s=0}\frac{L'}{L}(s, \mathrm{Sym}^n f )\frac{x^s}{s}.
\end{align*}
(Here, for convenience, $\rho$ can be either trivial or nontrivial.)

The goal of this section is to estimate the integral along $\gamma(T,U)$ in the positive direction.  We first estimate the integral along $S_1$.

\begin{lemma}
\label{S1-integral}
If $2\leq T\leq x$, then
\[
\frac{1}{2\pi i}\int_{S_1}\frac{L'}{L}(s, \mathrm{Sym}^n f )\frac{x^s}{s}~ds=O(x^{-U}U^{-1}T\log(T+U)).
\]
\end{lemma}
\begin{proof}
For weight 2 newforms, this is computed in \cite{Murty}.  The computation is exactly the same when $k>2$.
\end{proof}

We now estimate the integral along $S_2\cup S_3$.  Define

\begin{align*}
I_{1}( \mathrm{Sym}^n f ,x,T,U)=\frac{1}{2\pi i}\int_{-U}^{-1/4}&\left(\frac{x^{\sigma-iT}}{\sigma-iT}\frac{L'}{L}(\sigma-iT, \mathrm{Sym}^n f )\right.\\
&\left.-\frac{x^{\sigma+iT}}{\sigma+iT}\frac{L'}{L}(\sigma+iT, \mathrm{Sym}^n f )\right)~d\sigma
\end{align*}
and
\begin{align*}
I_2( \mathrm{Sym}^n f ,x,T)=\frac{1}{2\pi i}\int_{-1/4}^{\sigma_{0}}&\left(\frac{x^{\sigma-iT}}{\sigma-iT}\frac{L'}{L}(\sigma-iT, \mathrm{Sym}^n f )\right.\\
&\left.-\frac{x^{\sigma+iT}}{\sigma+iT}\frac{L'}{L}(\sigma+iT, \mathrm{Sym}^n f )\right)~d\sigma.
\end{align*}

To estimate $I_{1}( \mathrm{Sym}^n f ,x,T,U)$, we have the following.

\begin{lemma}
\label{S2-estimate}
Let $n\in\mathbb{Z}$ be positive.  If $s=\sigma+iT$, $\sigma\leq-1/4$, and $\left<s\right>\geq1/4$, then
\begin{align*}
\left|\frac{L'}{L}(s, \mathrm{Sym}^n f )\right|&\leq (n+1) (2\log (2 \left| s\right| +n+9)+2\log \mathfrak{q}(f)+25)\\
&+16 \log (2 \left| s\right| +n+9)+10.
\end{align*}
\end{lemma}
\begin{proof}
Logarithmic differentiation yields
\[
-\frac{L'}{L}(s, \mathrm{Sym}^n f )=\frac{n}{2}\log(N)+\frac{L'}{L}(1-s, \mathrm{Sym}^n f )+\frac{\gamma'}{\gamma}(s, \mathrm{Sym}^n f )+\frac{\gamma'}{\gamma}(1-s, \mathrm{Sym}^n f ).
\]
If $\sigma\leq-1/4$, then $\re(1-s)\geq5/4$.  Since $1-s$ is to the right of the critical strip, comparison with the Dirichlet series definition yields $|\frac{L'}{L}(1-s, \mathrm{Sym}^n f )|\leq4(n+1)$.  Since $\re(1-s)\geq5/4$ and $\re(s)\leq-1/4$, our hypotheses ensure that the conditions of Lemma \ref{log-deriv-gamma-factor} are satisfied at both $s$ and $1-s$.  The desired result now follows from applying Lemma \ref{log-deriv-gamma-factor}.
\end{proof}
We now estimate $I_1( \mathrm{Sym}^n f ,x,U,T)$.
\begin{lemma}
\label{I1}
If $x,T\geq4$, then $|\lim_{U\to\infty}I_1( \mathrm{Sym}^n f ,x,U,T)|$ is bounded above by
\begin{align*}
\frac{2}{\pi  T x^{1/4} \log (x)} [(n+1) (\log n+\log\mathfrak{q}(f)&+\log (T+1)+16)\\
&+ 200 \log (n)+200 \log (T+1)+883].
\end{align*}
\end{lemma}
\begin{proof}
Since $T\geq2$, the hypotheses of Lemma \ref{S2-estimate} are satisfied.  Using the approximations $|\frac{x^{\sigma+iT}}{\sigma+iT}|\leq\frac{x^\sigma}{T}$ and $\log(|\sigma+iT|)\leq\log(|\sigma|+2)+\log(|T|+2)$, we obtain the desired upper bound for $I_1( \mathrm{Sym}^n f ,x,U,T)$ by estimating the integrand with Lemma \ref{S2-estimate}.
\end{proof}

Finally, we estimate $I_{2}( \mathrm{Sym}^n f ,x,T)$.

\begin{lemma}
\label{I2}
If $-1/4\leq\sigma\leq \sigma_0$ and $6\leq T\leq x$, then $\left|I_{2}( \mathrm{Sym}^n f ,x,T)\right|$ is bounded above by
\begin{align*}
\frac{3x(n+1)}{T}&\left(20 \log \mathfrak{q}(f)+20  \log (n+T+5)+\frac{2 \log (n+T+5)+\log N+3}{\log x}\right)\\
&+\frac{3x}{T}\left(156 \log (n+T+5)+\frac{7 \log (n+T+5)+7}{\log x}\right).
\end{align*}
\end{lemma}
\begin{proof}
Using Lemma \ref{aux2}, we have
\begin{align*}
&\left|I_2( \mathrm{Sym}^n f ,x,T) \right. \\
                              & \left.-\frac{1}{2\pi i}\int_{-1/4}^{\sigma_{0}} \left(\frac{x^{\sigma-iT}}{\sigma-iT}\sum_{|\gamma+T|\leq 1}\frac{1}{\sigma-iT-\rho}-\frac{x^{\sigma+iT}}{\sigma+iT}\sum_{|\gamma-T|\leq 1}\frac{1}{\sigma+iT-\rho}\right)d\sigma\right|\\
&\leq\frac{3 x}{T \log (x)}((n+1) (2 \log (n+T+5)+\log N+3)+7 \log (n+T+5)+7).
\end{align*}
Using the residue theorem, one has that if $\rho=\frac{1}{2}+i\gamma$ and $\gamma\neq T$, then
\[
\left|\int_{-1/4}^{\sigma_0} \frac{x^{\sigma+iT}}{(\sigma+iT)(\sigma+iT-\rho)}~d\sigma\right|\leq(\sigma_0+3)\frac{x^{\sigma_0}}{(T-1)(\sigma_0-\frac{1}{2})}\leq\frac{22x}{T}.
\]
Thus
\begin{align*}
|I_2( \mathrm{Sym}^n f ,x,T)|&\leq\frac{3 x}{T \log (x)}[(n+1) (2 \log (n+T+5)+\log N+3)\\
&+7 \log (n+T+5)+7]+\frac{44x}{T}N_{ \mathrm{Sym}^n f }(T).
\end{align*}
The claimed result now follows from Lemma \ref{zero-count}.
\end{proof}

\section{The Explicit Formula}
\label{sec:explicit_formula}

We now write an explicit formula for $\psi_{ \mathrm{Sym}^n f }(x)$ as a sum over the zeros of $L(s, \mathrm{Sym}^n f )$.  We have shown that for $n\geq1$,
\[
\left|\psi_{ \mathrm{Sym}^n f }(x)-\left(-\sum_{\substack{\rho = \beta+i\gamma \\ \rho\neq0 \\ |\gamma|\leq T}}\frac{x^\rho}{\rho}-\mathrm{Res}_{s=0}\frac{L'}{L}(s, \mathrm{Sym}^n f )\frac{x^s}{s}\right)\right|\leq \mathcal{E},
\]
where $\rho=\beta+i\gamma$ is a zero of $L(s, \mathrm{Sym}^n f )$ (either trivial or nontrivial) and $\mathcal{E}$ is the sum of the upper bounds in Lemmata \ref{psi-estimate}, \ref{I1}, and \ref{I2}.

\subsection{The residue at $s=0$}

When $4 \nmid n$, $L(s, \mathrm{Sym}^n f )$ is nonzero at $s = 0$, while when
$4 | n$, since $\Gamma(\frac{s}{2})$ arises as part of $\gamma(s, \mathrm{Sym}^n f )$,
$L(0, \mathrm{Sym}^n f ) = 0$. In this case, the Laurent expansion of $-\frac{L'}{L}(s, \mathrm{Sym}^n f )$
at $s = 0$ is
\[
  -\frac{L'}{L}(s, \mathrm{Sym}^n f ) \frac{x^{s}}{s} = \frac{1}{s^{2}}
  + \frac{C - \log(x)}{s} + \cdots,
\]
for some constant $C$. Define
\[
  C_{ \mathrm{Sym}^n f } = \begin{cases}
    -\frac{L'}{L}(0, \mathrm{Sym}^n f ) & \text{ if } 4 \nmid n \\
    \lim_{s \to 0} -s \frac{L'}{L}(s, \mathrm{Sym}^n f ) + \log(x) & \text{ if } 4 | n.
\end{cases}
\]
It follows that 
\[
-\mathrm{Res}_{s=0} \frac{L'}{L}(s, \mathrm{Sym}^n f ) \frac{x^{s}}{s}=\begin{cases}
C_{ \mathrm{Sym}^n f }&\mbox{if $4\nmid n$,}\\
C_{ \mathrm{Sym}^n f }-\log x&\mbox{if $4\mid n$.}
\end{cases}
\]

Since the Dirichlet coefficients of $L(s, \mathrm{Sym}^n f )$ are real, it follows
that $C_{ \mathrm{Sym}^n f }$ is real and we derive the formula (valid when $4 \nmid n$)
of
\[
  C_{ \mathrm{Sym}^n f } := -\frac{L'}{L}( \mathrm{Sym}^n f , 0) = \frac{n}{2} \log N + \frac{\gamma'}{\gamma}( \mathrm{Sym}^n f , 0)
  - 1 + \sum_{\rho = 1/2 + it} \frac{2}{t^{2} + 4}.
\]
This can be bounded using the explicit form of the gamma factors
\eqref{gamma-factor} and using Lemma~\ref{zero-count}. A slight modification
is needed when $4 | n$.

Regardless of whether $4 \mid n$ or not, we have that
\[
  |C_{ \mathrm{Sym}^n f }| \leq \frac{n}{2} \log(N\pi)
  + 3.1 (n+1) \log\mathfrak{q}(f)+ \frac{n+1}{2} \log(n(k-1)) + 25.1 \log(n+1) + 52.
\]

\subsection{The sum over zeros}

The trivial zeros of $L(s, \mathrm{Sym}^n f )$ occur at the poles of the gamma factors.  When $n$ is odd, these poles are at $s=-m-(j+1/2)(k-1)$, where $0\leq j\leq\frac{n+1}{2}$ and $m$ is a nonnegative integer.  When $n$ is even, these poles are at $s=-2m-\frac{n}{2}\bmod 2$ (the case where $s=0$ is already accounted for) and $s=-m-j(k-1)$, where $1\leq j\leq\frac{n}{2}$.  These zeros have multiplicity bounded above by $1+\frac{n}{2}$, and so
\[
\sum_{\textup{$\rho\neq 0$ trivial}}\left|\frac{x^\rho}{\rho}\right|\leq\frac{n+2}{2}\sum_{m=1}^{\infty}\frac{x^{-m/2}}{m/2}\leq\frac{3n}{\sqrt{x}}.
\]

We now estimate the sum over nontrivial zeros.  By the assumption of GRH, we have that if $\rho=\beta+i\gamma$ is a nontrivial zero of $L(s, \mathrm{Sym}^n f )$, then $\beta=\frac{1}{2}$.  Thus
\begin{align*}
\sum_{\substack{\rho = \frac{1}{2}+i\gamma \\ |\gamma|\leq T}}\left|\frac{x^\rho}{\rho}\right|&\leq\sqrt{x}\sum_{\substack{\rho = \frac{1}{2}+i\gamma \\ |\gamma|\leq 4}}2+\sum_{\substack{\rho = \frac{1}{2}+i\gamma \\ 4<|\gamma|\leq T}}\frac{1}{|\gamma|}\\
&\leq2\sqrt{x}\sum_{j=0}^3N_{ \mathrm{Sym}^n f }^*(j)+\sqrt{x}\sum_{4\leq j\leq T}\frac{N_{ \mathrm{Sym}^n f }^*(j)}{j}.
\end{align*}

By choosing $T=5000\sqrt{x}$ and collecting the estimates for the sum over zeros, the residue at $s=0$, and $\mathcal{E}$, we see that if $x\geq5\times 10^5$, then $|\psi_{ \mathrm{Sym}^n f }(x)|$ is bounded above by
\begin{align}
\label{explicit-formula-bound}
|\psi_{ \mathrm{Sym}^n f }(x)|&\leq n \log (n) \left(\frac{5}{12} \sqrt{x} \log x +13 \sqrt{x}\right)\notag\\
&+n \left(\frac{8}{75} \sqrt{x} (\log x)^2+5 (\log\mathfrak{q}(f))\sqrt{x} \log x \right)\notag\\
&+12\sqrt{x}(\log x)(\log n)+3\sqrt{x}(\log x)^2\log\mathfrak{q}(f).
\end{align}
(We note that in the process of simplifying the upper bound for $|\psi_{ \mathrm{Sym}^n f }(x)|$, we use the fact that for non-CM newforms of squarefree level, the smallest value attained by $N(k-1)$ is $11$.)

\section{Proof of Propositions \ref{key-estimates-ab} and \ref{key-estimates-zero}}
\label{proof-ab}

We now prove Proposition \ref{key-estimates-ab}, from which Theorem \ref{main-theorem} was deduced in Section \ref{sec:proofs_of_main_theorems}.

\begin{proof}[Proof of Proposition \ref{key-estimates-ab}]
We first estimate the difference $\left|\psi_{ \mathrm{Sym}^n f }(x)-\theta_{ \mathrm{Sym}^n f }(x)\right|,$ where
\[
\theta_{ \mathrm{Sym}^n f }(x)=\sum_{p\leq x}U_n(\cos \theta_p )\log p.
\]
If $p\nmid N$, then $\Lambda_{ \mathrm{Sym}^n f }(p)=U_n(\cos \theta_p )$.  If $j=p^m$ for some $m\geq1$ and some $p\mid N$, then $|\Lambda_{ \mathrm{Sym}^n f }(j)|\leq p^{-mn/2}\log p$.  For all positive integers $j$, we have $|\Lambda_{ \mathrm{Sym}^n f }(j)|\leq (n+1)\Lambda(j)$.  Therefore, using Rosser and Schoenfeld's \cite{RS} bound of $\sum_{p\leq x}\log p\leq 1.001102x$ for all $x\geq0$, we have that if $x\geq 6$, then
\begin{align*}
\left|\psi_{ \mathrm{Sym}^n f }(x)-\theta_{ \mathrm{Sym}^n f }(x)\right|&\leq (n+1)\sum_{\substack{p^m\leq x \\ m\geq2 \\ p\nmid N}}\log p+\sum_{\substack{p^m\leq x \\ p\mid N}}p^{-mn/2}\log p+(n+1)\sum_{p\mid N}\log p\\
&\leq1.01(n+1)\sqrt{x}+(n+6)\log N.
\end{align*}
We now observe using partial summation that
\[
\sum_{p\leq x}U_n(\cos \theta_p )=\frac{\theta_{ \mathrm{Sym}^n f }(x)}{\log(x)}+\int_{2}^{x}\frac{\theta_{ \mathrm{Sym}^n f }(t)}{t(\log t)^2}dt,
\]
and the claimed bound in Proposition \ref{key-estimates-ab} now follows from the bound \eqref{explicit-formula-bound}.
\end{proof}

Because of the similarity between the proofs of Propositions \ref{key-estimates-ab} and \ref{key-estimates-zero}, we will only sketch the proof of Proposition \ref{key-estimates-zero}.

\begin{proof}[Sketch of the proof of Proposition \ref{key-estimates-zero}]

We will assume the notation in Section 3.  We begin with a smooth version of the explicit formula \cite[Lemma 3.3]{Rouse}, which states that if $G_x(s)=\int_0^\infty g_x(y)y^{s-1}dy$, then for $n\geq0$ and $x>1$, we have
\[
\sum_{j=1}^{\infty}\Lambda_{ \mathrm{Sym}^n f }(j)g_x(j)=\delta_{n,0}G_x(1)-\sum_{\rho}G_x(\rho),
\]
where $\rho$ is a zero of $L(s, \mathrm{Sym}^n f )$.  The trivial zeros yield a contribution of size $O(n)$, and for each nontrivial zero $\rho=\frac{1}{2}+i\gamma$ of $L(s, \mathrm{Sym}^n f )$, it follows from a change of variables that if $h(y)=2\pi g(e^{2\pi y})e^{\pi y}$, then
\[
|G_x(\rho)|=\sqrt{x}|\hat{h}(-\gamma)|\ll \frac{\sqrt{x}}{1+|\gamma|^2},
\]
where $\hat{h}$ denotes the Fourier transform.  Using Lemma \ref{zero-count} and the definition of $G_x(1)$, we then have that
\[
\sum_{j=1}^{\infty}\Lambda_{ \mathrm{Sym}^n f }(j)g_x(j)=\delta_{n,0}x\int_{0}^{\infty}g(t)dt+O(\sqrt{x}[n\log n+n \log\mathfrak{q}(f)]).
\]
The difference between this sum and the sum in Proposition \ref{key-estimates-zero} is $O(n[\sqrt{x}+\log\mathfrak{q}(f)])$, which is obtained by estimating the contribution from prime powers as in the proof of Proposition \ref{key-estimates-ab}.  The result now follows.
\end{proof}

\section{Proofs of Theorems \ref{densities} and \ref{tau-QF}}
\label{sec:lehmer_densities}

We now want to use Theorem \ref{count-primes-zero} to give a lower bound for the density of positive integers $n$ for which $a_f(n)\neq0$.  Using the bound for $ \pi_{0,f}(x) $ in Theorem \ref{count-primes-zero}, we can produce an explicit lower bound for the density of positive integers $n$ such that $a_f(n)\neq0$.  This allows us to address analogues of Lehmer's question (cf. \cite{Leh}) for any newform $f$ with squarefree level and trivial character, assuming Conjecture \ref{automorphy-GRH}.

\begin{lemma}
\label{densitylowbound}
Assume the above notation.  Let $x_0>3$, and define
\[
\omega_{0,f}(x,x_0)=\sum_{j=1}^{1+\lfloor \log_2(x/x_0) \rfloor}\pi_{0,f}\left(\frac{x}{2^j}\right).
\]
Then
\[
\prod_{a_f(p)=0}\left(1-\frac{1}{p+1}\right)>\exp\left(-\int_{x_0}^\infty\frac{\omega_{0,f}(x,x_0)}{x^2+x}~dx\right)\prod_{\substack{p\leq x_0\\a_f(p)=0}}\left(1-\frac{1}{p+1}\right).
\]
\end{lemma}

\begin{proof}
This follows directly from the definition of $\pi_{0,f}(x)$ given by \eqref{pi-zero}, an application of Abel summation to the log of the product in \eqref{Serre-density}, and the straightforward computation
\[
\exp\left(-\int_2^{x_0}\frac{\#\{p\leq x:a_f(p)=0\}}{x^2+x}~dx\right)=\left(1+\frac{1}{x_0}\right)^m\prod_{j=1}^m \left(1-\frac{1}{p_j+1}\right),
\]
where $p_1,\ldots,p_m$ are the primes less than $x_0$ for which $a_f(p)=0$.
\label{density-theorem}
\end{proof}

To apply the preceding lemma, it is necessary to bound the number of small primes $p$ for which $a_f(p) = 0$. On page 168 of \cite{Bosman}, Bosman repeats Serre's observation that if $\tau_{12}(p) = 0$, then
$p = hM - 1$, where $M = 3094972416000$ and $h\geq1$. Moreover,
$h \equiv 0, 30, \text{ or } 48 \pmod{49}$ and $h+1$ is a quadratic residue
modulo $23$. These facts will allow us to bound the density above.

In order to obtain results of a similar quality for the other level 1 newforms, we need analogues of the
congruences for $\tau_{12}(n)$ given in \cite{Swinnerton-Dyer} for the
higher weight level 1 newforms. We state these congruences here.

\begin{theorem}
\label{level1cong}
For the weight $16$ form we have
{\small
\begin{align*}
  \tau_{16}(n) &\equiv 6497 \sigma_{15}(n) \pmod{2^{13}}
  \text{ if } n \equiv 7 \pmod{8},\\
  \tau_{16}(n) &\equiv n^{813} \sigma_{2763}(n) \pmod{3^{8}}
  \text{ if } n \equiv 2 \pmod{3},\\
  \tau_{16}(n) &\equiv n^{17} \sigma_f(n) \pmod{5^{2}}
  \text{ if } \gcd(n,5) = 1,\\
  \tau_{16}(n) &\equiv n^{85} \sigma_{139}(n) \pmod{7^{3}}
  \text{ if } \gcd(n,7) = 1,\\
  \tau_{16}(n) &\equiv n \sigma_{3}(n) \pmod{11}
  \text{ if } \gcd(n,11) = 1,\\
  \tau_{16}(n) &\equiv n^{2} \tau_{12}(n) \pmod{13},\\
  \tau_{16}(n) &\equiv 0 \pmod{31} \text{ if }
  \legen{n}{31} = -1,\\
  \tau_{16}(n) &\equiv \sigma_{15}(n) \pmod{3617}.\\
\end{align*}}%
For the weight $18$ form we have
{\small
\begin{align*}
  \tau_{18}(n) &\equiv 865 \sigma_{17}(n) \pmod{2^{13}} \text{ if }
  n \equiv 7 \pmod{8},\\
  \tau_{18}(n) &\equiv n^{117} \sigma_{269}(n) \pmod{3^{6}} \text{ if }
  n \equiv 2 \pmod{3},\\
  \tau_{18}(n) &\equiv n^{22} \sigma_{73}(n) \pmod{5^{3}} \text{ if }
  \gcd(n,5) = 1,\\
  \tau_{18}(n) &\equiv \begin{cases} n \sigma_{3}(n) \pmod{7}\\
  n \sigma_{15}(n) \pmod{7^{2}} \text{ if }
  \legen{n}{7} = -1,\\
  \end{cases}\\
  \tau_{18}(n) &\equiv \begin{cases}n  \sigma_{5}(n) \pmod{11}\\
  n \sigma_{15}(n) \pmod{11^{2}} \text{ if } \legen{n}{11}
  = -1,\\
  \end{cases}\\
  \tau_{18}(n) &\equiv n \sigma_{3}(n) \pmod{13} \text{ if } \gcd(n,13) = 1,\\
  \tau_{18}(n) &\equiv \sigma_{17}(n) \pmod{43867}.
\end{align*}}%
For the weight $20$ form we have
{\small
\begin{align*}
    \tau_{20}(n) &\equiv 2945 \sigma_{19}(n) \pmod{2^{15}} \text{ if }
  n \equiv 7 \pmod{8},\\
  \tau_{20}(n) &\equiv n^{207} \sigma_{91}(n) \pmod{3^{6}} \text{ if }
  n \equiv 2 \pmod{3},\\
  \tau_{20}(n) &\equiv n^{6} \sigma_{7}(n) \pmod{5^{2}} \text{ if }
  \gcd(n,5) = 1,\\
  \tau_{20}(n) &\equiv
  \begin{cases} n^{2} \sigma_{3}(n) \pmod{7}\\
  n^{2} \sigma_{15}(n) \pmod{7^{2}} \text{ if } \legen{n}{7}
  = -1,
  \end{cases}\\
  \tau_{20}(n) &\equiv n \sigma_{7}(n) \pmod{11},\\
  \tau_{20}(n) &\equiv n \sigma_{5}(n) \pmod{13},\\
  \tau_{20}(n) &\equiv n^{2} \tau_{16}(n) \pmod{17},\\
  \tau_{20}(n) &\equiv \sigma_{19}(n) \pmod{283},\\
  \tau_{20}(n) &\equiv \sigma_{19}(n) \pmod{617}.\\
\end{align*}}%
For the weight $22$ form we have
{\small
\begin{align*}
  \tau_{22}(n) &\equiv 3969 \sigma_{21}(n) \pmod{2^{15}} \text{ if }
  n \equiv 7 \pmod{8},\\
  \tau_{22}(n) &\equiv n^{3} \sigma_{15}(n) \pmod{3^{8}} \text{ if }
  n \equiv 2 \pmod{3},\\
  \tau_{22}(n) &\equiv n^{7} \sigma_{7}(n) \pmod{5^{2}},\\
  \tau_{22}(n) &\equiv n^{26} \sigma_{11}(n) \pmod{7^{2}} \text{ if }
  \gcd(n,7) = 1,\\
  \tau_{22}(n) &\equiv \tau_{12}(n) \pmod{11},\\
  \tau_{22}(n) &\equiv n \sigma_{7}(n) \pmod{13},\\
  \tau_{22}(n) &\equiv n \sigma_{3}(n) \pmod{17},\\
  \tau_{22}(n) &\equiv n^{2} \tau_{18}(n) \pmod{19},\\
  \tau_{22}(n) &\equiv \sigma_{21}(n) \pmod{131},\\
  \tau_{22}(n) &\equiv \sigma_{21}(n) \pmod{593}.\\
\end{align*}}%
For the weight $26$ form we have
{\small
\begin{align*}
  \tau_{26}(n) &\equiv 545 \sigma_{25}(n) \pmod{2^{13}} \text{ if }
  n \equiv 7 \pmod{8},\\
  \tau_{26}(n) &\equiv n^{171} \sigma_{169}(n) \pmod{3^{6}} \text{ if }
  n \equiv 2 \pmod{3},\\
  \tau_{26}(n) &\equiv n^{6} \sigma_{13}(n) \pmod{5^{2}},\\
  \tau_{26}(n) &\equiv
  \begin{cases} n^{2} \sigma_{3}(n) \pmod{7}\\
  n^{2} \sigma_{21}(n) \pmod{7^{3}} \text{ if }
  \legen{n}{7} = -1,\\
  \end{cases}\\
  \tau_{26}(n) &\equiv n \sigma_{3}(n) \pmod{11},\\
  \tau_{26}(n) &\equiv n \tau_{12}(n) \pmod{13},\\
  \tau_{26}(n) &\equiv n \sigma_{7}(n) \pmod{17},\\
  \tau_{26}(n) &\equiv n \sigma_{5}(n) \pmod{19},\\
  \tau_{26}(n) &\equiv n^{2} \tau_{22}(n) \pmod{23},\\
  \tau_{26}(n) &\equiv \sigma_{25}(n) \pmod{657931}.
\end{align*}}
\end{theorem}
\begin{proof}
For brevity, we only provide proofs of the congruence for the weight $16$ form.  Twisting by quadratic Dirichlet characters mod $8$ shows that
{\small
\[
\sum_{n \equiv 7 (\text{mod}~8)} \tau_{16}(n) q^{n},\sum_{n \equiv 7 (\text{mod}~8)}
\sigma_{15}(n) q^{n}\in M_{16}(\Gamma_{0}(64)),
\]}%
by Theorem~\ref{twisting}. A computation
shows that the first congruence holds for $n \leq 128$ and Sturm's theorem
implies that it is true for all $n$.

For the second congruence, we start by taking $E_{2764}(z)$ and twisting
to obtain the form $\sum_{n \equiv 2 \pmod{3}} \sigma_{2763}(n) q^{n} \in
M_{2764}(\Gamma_{0}(9))$. Applying $\theta$ 813 times, and using
Lemma~\ref{Primepower} with $N = 3^{9}$ gives that there is a form
in $M_{4390}(\Gamma_{0}(3^{9}))$ congruent to $\sum_{n \equiv 2 \pmod{3}}
n^{813} \sigma_{2763}(n) q^{n}$ modulo $3^{8}$. Taking the form
$\Delta_{12}(z) E_{4}(z) E_{\phi(3^{8})}(z)$ and twisting it to isolate
the residue class $2$ mod $3$ shows that $\sum_{n \equiv 2 \pmod{3}}
\tau_{16}(n) q^{n}$ is also congruent mod $3^{8}$ to a form in
$M_{4390}(\Gamma_{0}(3^{9}))$. We now check the congruence
up to $n = 9600930$ (requiring about 1 hour using Magma \cite{Magma} V2.18) and
invoke Sturm's theorem.

For the third congruence, we let $f(z) \in M_{36}(\Gamma_{0}(25))$ be a form
congruent to $\theta^{17}(E_{2,25})$ modulo $25$ and
$g(z) = E_{4}(z) \Delta_{12}(z) E_{20}(z)$. Then $f(z)  - f(z) | U(5) | V(5)$
and $g(z) - g(z) | U(5) | V(5)$ are both in $M_{36}(\Gamma_{0}(25))$
and we use Sturm's theorem to prove the congruence.

The fourth congruence is similar to the third, and
the fifth congruence follows from the fact that $\theta(E_{4})$ is congruent
(modulo $11$) to a form of weight $16$. The sixth congruence follows
from the fact that $\theta^{2}(\Delta_{12})$ is congruent to a form of weight
$40$. The seventh congruence was proven in \cite{Swinnerton-Dyer} and
the eighth is a consequence of the fact that $3617$ divides the numerator
of $B_{16}$.
\end{proof}

It is clear that if $\tau_{k}(p)=0$ for any $k\in\{12, 16, 18, 20, 22,26\}$, then $p$ must satisfy congruence conditions with a rather large modulus.  For certain ranges of $x$, counting the putative primes $p$ for which $\tau_k(p)=0$ is done more effectively using the variant of the Brun-Titchmarsh theorem proven by Montgomery and Vaughan \cite{MV} instead of Theorem \ref{count-primes-zero}.

\begin{theorem}
\label{thm:MV}
Let $x$ and $y$ be positive real numbers, and let $a$ and $q$ be relatively prime positive integers.  If $\pi(x;q,a)=\#\{p\leq x:p\equiv a\pmod q\},$ then
\[
\pi(x+y;q,a)-\pi(x;q,a)<\frac{2y}{\varphi(q)\log(y/q)}.
\]
\end{theorem}

Using Theorem \ref{count-primes-zero}, Lemma \ref{densitylowbound}, Theorem \ref{level1cong}, and Theorem \ref{thm:MV}, we prove Theorem \ref{densities}.

\begin{proof}[Proof of Theorem \ref{densities}]
For the case where $f\in S_2^{\textup{new}}(\Gamma_0(11))$, we compute a
polynomial $P_{\ell}(x)$ whose splitting field is the fixed field of the kernel
of the projective representation for $2 \leq \ell \leq 19$. Using the
factorization of these polynomials modulo primes $p$, we find that
there are precisely $17857$ primes $p \leq 10^{11}$ for which $a_f(p) = 0$, which enables us to compute
\[
\prod_{\substack{p\leq 10^{11} \\ a_f(p)=0}}\left(1-\frac{1}{p+1}\right)=0.8465247961\ldots
\]
Using Theorem \ref{count-primes-zero} and Lemma \ref{densitylowbound}, we obtain
\[
\exp\left(-\int_{10^{11}}^\infty\frac{\omega_{0,f}(x,10^{11})}{x^2+x}~dx\right)>0.9811342755\ldots,
\]
and the claimed bounds on $D_{f}$ now follow from Lemma \ref{densitylowbound}.

Using the congruences for $\Delta_{12}(z)$ and the computation of the
mod $11$, mod $13$, mod $17$ and mod $19$ Galois representations
by Bosman, we compute using PARI/GP \cite{PARIGP} that there are
precisely $1810$ primes $p < 10^{23}$ that satisfy the conditions
given by Serre, and for which $\tau_{12}(p) \equiv 0 \pmod{11 \times
13 \times 17 \times 19}$.  This allows us to compute
\[
\prod_{\substack{\tau_{12}(p)=0 \\ p\leq 10^{23}}}\left(1-\frac{1}{p+1}\right)>0.99999999999999999980399\ldots
\]
By aforementioned the work of Serre, if $\tau_{12}(p)=0$, then $p$ is in one of 33 possible residue classes modulo $M=23\times49\times3094972416000$.  Using Theorem \ref{thm:MV}, we have that for all $x\geq 10^{23}$,
\[
\#\{p\leq x:\tau_{12}(p)=0\}\leq 1810+\frac{x-10^{23}+2M}{\varphi(M)\log((x-10^{23}+2M)/M)}\times33.
\]
With this upper bound, we calculate
\[
\exp\left(-\int_{10^{23}}^{10^{60}}\frac{\#\{p\leq x:\tau_{12}(p)=0\}}{x^2+x}~dx\right)>0.99999999999984987\ldots
\]
Finally, using Theorem \ref{count-primes-zero}, we find that
\begin{align*}
\exp\left(-\int_{10^{60}}^{\infty}\frac{\#\{p\leq x:\tau_{12}(p)=0\}}{x^2+x}~dx\right)&\geq\exp\left(-\int_{10^{60}}^{\infty}\frac{\omega_{0,\Delta_{12}}(x,10^{60})}{x^2+x}~dx\right)\\
&>0.9999999999999961446\ldots
\end{align*}
Inserting the above work into Lemma \ref{densitylowbound}, we obtain the claimed bound for $D_{\Delta_{12}}$.

The higher weight cases are handled similarly. For the weight $16$ case we
also make use of the fact that the projective mod $59$ representation
has image isomorphic to $S_{4}$. This was conjectured by Serre \cite{SerreBourbaki416} and
Swinnerton-Dyer \cite[pg. 35]{Swinnerton-Dyer}.  The degree $4$ polynomial defining this
$S_{4}$ extension is misprinted there (and corrected in
\cite[pg. 149]{Swinnerton-DyerII}); the correct polynomial is $x^{4}
- x^{3} - 7x^{2} + 11x + 3$. Swinnerton-Dyer's conjecture was proven
by K. Haberland  \cite{Haberland} in 1982. This can also be
proven using the solvable base change result of Langlands and Tunnell
(in an identical way that it was applied in Wiles's proof of Fermat's
Last Theorem).
\end{proof}

\begin{proof}[Proof of Theorem \ref{tau-QF}]
Let $\gamma_1=259$, $\gamma_2=11920$, and $\gamma_3=1060864$.  An equivalent formulation of the theorem is that $\gamma_1\tau_{12}(n)+\gamma_2\tau_{12}(n/2)+\gamma_3\tau_{12}(n/4)=0$ if and only if $\tau_{12}(n)=0$.  Write $n=2^\alpha k$, where $k$ is a positive odd integer and $\alpha$ is a nonnegative integer.  Now,
{\small
\[
\gamma_1\tau_{12}(n)+\gamma_2\tau_{12}(n/2)+\gamma_3\tau_{12}(n/4)=\frac{\tau_{12}(n)}{\tau_{12}(2^\alpha)}(\gamma_1\tau_{12}(2^{\alpha})+\gamma_2\tau_{12}(2^{\alpha-1})+\gamma_3\tau_{12}(2^{\alpha-2})).
\]}%
It is proven in \cite{Leh} that if $p$ is prime and $\tau_{12}(p^\alpha)=0$, then $\tau_{12}(p)=0$ and $\alpha$ is odd.  Therefore, $\tau_{12}(2^\alpha)\neq0$ for all $\alpha$.  Define $f_\alpha=\gamma_1\tau_{12}(2^{\alpha})+\gamma_2\tau_{12}(2^{\alpha-1})+\gamma_3\tau_{12}(2^{\alpha-2}).$  It suffices to show that $f_\alpha\neq0$ for all positive integers $\alpha$.  This is clear for $\alpha=0,1$, so we may assume that $\alpha\geq2$.  By a straightforward calculation, if $\alpha\geq2$, then $f_\alpha=-24f_{\alpha-1}-2048f_{\alpha-2}.$  The closed-form solution to this recursive equation is
\[
f_\alpha=(\lambda_1\phi^\alpha+\lambda_2\overline{\phi}^\alpha)/\lambda_0,
\]
where $\lambda_0=1904(3-\sqrt{-119})$, $\lambda_1=1479408-370960\sqrt{-119}$, $\lambda_2=797895-388141\sqrt{-119}$, and $\phi=-12+4\sqrt{-119}$.

Let $K=\mathbb{Q}(\sqrt{-119})$.  It now suffices to prove $\frac{-\lambda_1}{\lambda_2}$ is not of the form $(\frac{\beta_1}{\beta_2})^\alpha$, where $\alpha\geq2$ and $\beta_1$ and $\beta_2$ are in the ring of integers $\mathcal{O}_K$.  We prove the contrapositive.  If $\frac{-\lambda_1}{\lambda_2}=(\frac{\beta_1}{\beta_2})^\alpha$, then $-\lambda_1 \beta_2^\alpha=\lambda_2 \beta_1^\alpha$.  Consider the ideals $I_1=(-\lambda_1\beta_2^\alpha)$ and $I_2=(\lambda_1 \beta_1^\alpha)$ in $\mathcal{O}_K$.  By a computation in Magma, there exists a prime ideal $\mathfrak{p}$ in $\mathcal{O}_K$ such that $3,\lambda_2\in\mathfrak{p}$ and $-\lambda_1\notin\mathfrak{p}$.  Furthermore, the power of $\mathfrak{p}$ in the prime factorization of the ideal $(\lambda_2)$ is 1.  The power of $\mathfrak{p}$ in the prime factorization of $I_1$ is $\eta_1\alpha$, where $\eta_1$ is the power of $\mathfrak{p}$ in the prime factorization of the ideal $(\beta_2)$.  The power of $\mathfrak{p}$ in the prime factorization of $I_2$ is $\eta_2\alpha+1$, where $\eta_2$ is the power of $\mathfrak{p}$ in the prime factorization of the ideal $(\beta_1)$.  Because $I_1=I_2$, we have $\eta_1 \alpha=\eta_2\alpha+1$.  Thus $0\equiv1\pmod{\alpha}$, so $\alpha=1$.
\end{proof}

\bibliographystyle{plain}
\bibliography{Paper5}
\end{document}